\documentclass[12pt]{article}
\usepackage{graphicx} % Required for inserting images
\usepackage{amsmath}
\usepackage{amsthm}
\usepackage{amssymb}
\usepackage{geometry}
\usepackage{mathrsfs}
\usepackage{lmodern}
\usepackage{upgreek}
\usepackage{enumitem}
\usepackage{scalerel}
\usepackage[
backend=biber,
sorting=nyt
]{biblatex}
\usepackage{xcolor}
\usepackage{hyperref}
\hypersetup{
pdfborder = {0 0 0},
colorlinks,
linkcolor=blue!40!black,
citecolor=blue!40!black,
urlcolor=blue!40!black
}

\usepackage{csquotes}
\usepackage[T1]{fontenc} 
\usepackage{dsfont}
\newcommand{\RR}{\mathbb R}
\newcommand{\R}{\mathcal R}

\newcommand{\J}{\mathcal J}
\newcommand{\EE}{\mathbb E}
\newcommand{\TT}{\mathbb T}
\newcommand{\PP}{\mathbb P}

\newcommand{\1}{\mathds{1}}
\newcommand{\ZZ}{\mathbb Z}
\renewcommand{\d}{\mathrm d}
\renewcommand{\i}{\mathrm i}

\newcommand{\distribution}{\stackrel{\scaleto{\mathrm{law}}{4pt}}{\longrightarrow}}
\newcommand{\Var}{\mathrm{Var}}

\newcommand{\eqdef}{\stackrel{\scaleto{\mathrm{def}}{3.5pt}}{=}}

\newtheorem{theorem}{Theorem}
\newtheorem*{theorem*}{Theorem}

\newtheorem{proposition}{Proposition}
\newtheorem{lemma}{Lemma}
\newtheorem*{remark}{Remark}
\newtheorem{corollary}{Corollary}
\newtheorem{conjecture}{Conjecture}

\numberwithin{equation}{section}

\title{On the range of competing random walks}
\author{Maxence Baccara$^*$}

\begin{document}

\maketitle
\let\thefootnote\relax\footnotetext{$^{*}$ CMAP \& CPHT \& MERGE, École polytechnique, Institut Polytechnique de Paris, 91120 Palaiseau, France. \\ Email : maxence.baccara@polytechnique.edu}

\begin{abstract}
    We consider $N$ independent random walks $X^1,\dots,X^N$ in the lattice $\ZZ^d$ and prove limit theorems for the competitive range $\R_n^k$ of the $k$-th random walk $X^k$, which corresponds to the number of distinct sites that it has discovered before any of the other $X^\ell$, $\ell\ne k$, up to time $n$. This is a natural object to study foraging mechanisms in population ecology, in which context it is also natural to ask how the effect of competition for the access to resources affects the number of resources consumed by each individual. We work with random walks in the domain of attraction of a $\beta$-stable law and focus on the regime $d/\beta\in[1,3/2)$, in which classical results for the range show that the fluctuations are described by the renormalized self-intersection local time of the limiting process. We establish a central limit theorem in which a competition term emerges, thus answering the two previous questions we asked. We end the paper with a brief discussion on the remaining regimes $d/\beta\ge3/2$, in which the fluctuations are Gaussian and are not affected by the competition, and $d/\beta<1$ in which no strong law of large numbers holds and we expect the effect of the competition to strongly affect the first-order asymptotics.
\end{abstract}

\tableofcontents

\section{Introduction}

\subsection{Main objects and literature review}

Let $N$ be a fixed positive integer greater than or equal to $2$. Let $X^1,\dots,X^N$ be $N$ independent and identically distributed random walks in $\ZZ^d$ defined on the probability space $\left(\Omega,\mathcal F,\PP,\left(\theta_n\right)_{n\ge0}\right)$, where $\theta_n$ denotes the $n$-th shift operator. For any $k\in\{1,\dots,N\},x\in\ZZ^d$, we let $T_x^k$ denote the first hitting time of $x$ by $X^k$. In this work, motivated by questions in ecology of populations stemming from \cite{Bansaye2024}, we shall study the quantities
\begin{equation*}
    \R_n^k\eqdef\sum_{x\in\ZZ^d}\1\left(T_x^k\le n,T_x^k<\underset{j\ne k}\min \ T_x^j\right),\quad k\in\{1,\dots,N\}.
\end{equation*}
Simply put, $\R_n^k$ denotes the number of distinct sites that were visited by $X^k$ before they were visited by any other of the $X^j$, $j\ne k$. In what follows, we shall refer to $\R^k$ as the competitive range of $X^k$, as opposed to the range, which is simply the number of distinct sites visited by $X^k$ regardless of whether or not they were also visited by the other $X^j$, $j\ne k$, $i.e.$ the quantity
\begin{equation*}
    \left|X^k(0,n)\right|\eqdef\left|\{X_j^k, j\in[\![0,n]\!]\}\right|.
\end{equation*} 

We consider the competitive range in order to study foraging mechanisms in a simple prey-predator model, which we now describe. Working in $\ZZ^d$, we start with $N$ independent random walks which each represent a different predator at the origin and with one prey on every other site of the lattice. Then, at integer times, we let each predator move in the lattice according to its jump distribution. If it lands on a site occupied by a prey and no other predator, then the predator consumes the prey instantaneously. If two or more predators land on a site that is occupied by prey, then the prey in question is removed from the lattice but is not counted as being consumed by any of the predators that first discovered it. We could also have supposed that if two predators discover a site occupied by prey, then the prey is counted towards both their competitive ranges. However, asymptotically, these two conventions are equivalent since the number of collisions of two independent random walks ($i.e.$ the number of $\ell\in\{1,\dots,n\}$ such that $X_\ell=\widetilde X_\ell$) is very small in front of the number of sites that the random walks in question visit. For simplicity, we therefore choose the former. Finally, if a predator lands on a site that contains no prey, then nothing happens. As for the prey, we do not suppose that there is any motion or regeneration. The question we aim to answer is thus the following: after $n$ steps, where $n$ is large, how many prey were consumed by the $k$-th predator? In order to give a satisfying answer, we shall establish, under some hypothesis on the random walks $X^1,\dots, X^N$, a Law of Large Numbers (LLN) and a Central Limit Theorem (CLT) for the competitive ranges $\R^k$, $k\in\{1,\dots,N\}$. By this we understand respectively that
\begin{equation*}
    \R_n^k/\EE\left[\R_n^k\right]\underset{n\rightarrow\infty}{\overset{a.s.}\longrightarrow}1,
\end{equation*} 
and that 
\begin{equation*}
    \Var\left(\R_n^k\right)^{-1/2}\left(\R_n^k-\EE\left[\R_n^k\right]\right)\distribution Z,
\end{equation*}
where $Z$ is a non-degenerate random variable. In particular, we are interested in understanding how the effect of competition between individuals appears in our limit theorems.

The range $|X(0,n)|$ of a random walk $X$ is a classical object that has been extensively studied and is now well understood. It was first shown by Dvoretsky and Erdös (\cite{DvoretzkyErdos1951}) that in the case where $X$ is the simple random walk, then $|X(0,n)|$ satisfies the LLN in dimension $d\ge2$. It was then shown respectively by Kesten, Spitzer and Whitman (\cite{Spitzer1965}) and by Jain and Pruitt (\cite{JainPruitt1972}) that the LLN holds for any transient random walk and for recurrent two-dimensional random walks. The CLT was established for any random walk in $\ZZ^d$, $d\ge3$ by Jain and Pruitt (\cite{JainPruitt1971}, \cite{JainPruitt1974}), and then by Le Gall for centered $L^2$ walks in $\ZZ^2$ (\cite{LeGallIntersectionsMarchesAléatoires1986}). In the first case, the limiting distribution is Gaussian, and remarkably it is non-Gaussian for planar random walks with finite variance and the fluctuations are descirbed by a random variable that counts the number of self-intersections of Brownian motion. It was later shown by Le Gall and Rosen (\cite{LeGallRosenRangeOfStableRandomWalks1991}) that the LLN and CLT hold for any $\beta$-stable walk (see Section \ref{susbection:main results, strategy of proof}, assumption \ref{Assumption A1}) potentially under some additional hypothesis. They show that there are three regimes, depending on the ratio $d/\beta$. When $d/\beta\ge3/2$, the range has linear growth in the first order and Gaussian fluctuations at the diffusive scale, with some additional small correction when $d/\beta=3/2$. When $3/2>d/\beta\ge1$, the range also grows linearly when the walk is transient and is subject to some correction given by the truncated Green function otherwise, and the fluctuations are non-Gaussian and given by a random variable that counts the number of self-intersections of the stable process obtained when scaling the underlying random walk. Finally, when $d/\beta<1$, which in particular forces $d=1$, then there is not any LLN and instead the quantity $b(n)^{-d}|X(0,n)|$ converges in distribution to the Lebesgue measure of the set of points visited by the limiting stable process. Many other results regarding the range have been obtained in the literature, such as a Law of Iterated Logarithms and Large and Moderate deviations. For an accurate account of such results, one can consult the introduction of \cite{BassChenRosenModerateDeviations2009}. 

The present work is related to the finite-volume colouring or painting models studied in \cite{JuniorLucenaSilvaHilhorst1996}, \cite{Miller2010PaintingAG},
but the asymptotic problem considered here appears to be different from the existing literature. In \cite{JuniorLucenaSilvaHilhorst1996}, Gomes Jr., Lucena, da Silva and Hilhorst introduced a one-dimensional colouring
model in which two independent random walkers on a ring irreversibly colour each site according to which walker visits it first. Their analysis focuses on the final colouring of the ring,
including the one-point colouring probabilities and the number of interfaces.
Miller (\cite{Miller2010PaintingAG}) later studied a general finite-graph version of this painting problem. For two independent random walks on a sequence of graphs, he investigated the sets $A_i$ of vertices painted
first by each walk once the graph has been fully covered, and obtained sharp asymptotics for
$\Var(|A_i|)$ on $\ZZ^d$, $d\ge3$. By contrast, we consider competing walks in infinite volume and study the number of sites
discovered first by a given walk before a deterministic time $n$. Thus the competitive range
considered here is a time-dependent analogue of the final painted regions in finite-volume
models. This leads to different asymptotic questions, namely laws of large numbers and fluctuation limits in the stable scaling regimes. Competitive foraging mechanisms were also studied very recently in the physics literature \cite{plaud2026} for a broad class of transport processes. The authors characterize the fraction of resources consumed by a tagged walker and derive its asymptotic behaviour, including exact prefactors in the different fluctuation regimes.

\subsection{Main results, strategy of proof and some notation}
\label{susbection:main results, strategy of proof}

We now present the main results of the paper, namely a LLN and CLT for the competitive range, and give an overview of the strategy of proof used. In what follows we shall be working under the additional assumptions
\begin{enumerate}[label=(A\arabic*)]
    \item
    \label{Assumption A1}
    There is some $\beta\in(0,2]$ satisfying $d/\beta\in[1,3/2)$ such that   for any $k\in\{1,\dots,N\}$, $X^k$ is in the domain of attraction of a (strictly) $\beta$-stable process, $i.e.$ there is a $1/\beta$-regularly varying function $b$ (see Section \ref{section:Preliminary Results}) and a strictly $\beta$-stable process $U$ such that for each $k\in\{1,\dots,N\}$,
    \begin{equation*}
        \frac{X_n^k}{b(n)}\distribution U_1.
    \end{equation*}
    Furthermore, we suppose without loss of generality that $b$ is continuous and strictly increasing on $\RR_+$. 
    \item
    \label{Assumption A2}
    For each $k\in\{1,\dots,N\}$, the additive group generated by $\{x\in\ZZ^d,\ \PP\left(X_1^k=x\right)>0\}$ is the whole space $\ZZ^d$.
    \item
    \label{Assumption A3}
    The characteristic function of the $X_1^k$, noted $\varphi$, is continuously differentiable on $[-\pi,\pi]^d\backslash\{0\}$ and for all $x\in[-\pi,\pi]^d\backslash\{0\}$,
    \begin{equation*}
        |\nabla\varphi(x)|\lesssim\left(|x| b^{-1}(|x|^{-1})\right)^{-1}.
    \end{equation*}
\end{enumerate}

The main assumption here is \ref{Assumption A1}. Indeed, we work in the regime $d/\beta\in[1,3/2)$ in which the fluctuations of the range are given by the renormalized self-intersection local time of the stable process that is obtained as a scaling limit of the random walk. It is well known (see \cite{feller1991introduction}) that in order for a random walk to satisfy a CLT like the one in \ref{Assumption A1}, then it must necessarily be in the domain of attraction of a stable process. Furthermore, the fact that $b$ can be supposed continuous and strictly increasing can also be found in (\cite{feller1991introduction} $pp. 577-580$).

Assumptions \ref{Assumption A2} and \ref{Assumption A3} are used in \cite{LeGallRosenRangeOfStableRandomWalks1991} to establish a LLN and CLT for the range in the regime $1\le d/\beta<3/2$. The condition \ref{Assumption A2} (strong aperiodicity) is used to simplify calculations but is not essential. The technical condition \ref{Assumption A3} plays a more important role, since it allows the authors to obtain a limit theorem for the quantity $T_x/b^{-1}(|x|)$ as $x\rightarrow\infty$, which is essential to obtain scaling limits of the number of intersections of two independent random walks. In the special case $\beta>1$, \ref{Assumption A3} is automatically satisfied, see \cite{LeGallRosenRangeOfStableRandomWalks1991}, Proposition $5.4$.

For any $x,y\in\ZZ^d$, we define the Green function of the $X^k$ as 
\begin{equation*}
    G(x,y)\eqdef\EE_x\left[\sum_{j=0}^\infty\1\left(X_j^k=y\right)\right]=\sum_{j=0}^\infty\PP_x\left(X_j^k=y\right).
\end{equation*}
For any $x,y\in\ZZ^d$, $G(x,y)$ represents the average amount of times that $X^k$ visits $y$ if it is started from $x$. Note that when the random walk is recurrent, then $G\equiv\infty$, and so we also define the truncated Green function
\begin{equation*}
    G_{n}(x,y)\eqdef\sum_{j=0}^n\PP_x\left(X_j^k=y\right),\quad n\ge1,\quad x,y\in\ZZ^d,
\end{equation*}
which instead represents the average number of visits to $y$ started from $x$ before time $n$. In particular, we define
\begin{equation*}
    h(n)\eqdef G_{n}(0,0),\quad n\ge1.
\end{equation*}
If $X^k$ is transient, then $h(n)$ remains bounded and it is well-known that 
\begin{equation*}
    \lim_{n\rightarrow\infty} h(n)=\PP\left(\forall j\ge1, X_j^k\ne X_0^k\right)^{-1}\eqdef q^{-1}.
\end{equation*}
When $X^k$ is recurrent, which may only occur if $d=\beta$ (recall that $d/\beta\ge1)$, then we rather have 
\begin{equation}
\label{limit:trunctaed green's function asymptotics}
    h(n)\underset{n\rightarrow\infty}{\sim}p_1(0)\sum_{j=1}^n\frac{1}{b(j)^d},
\end{equation}
where $p_t$ denotes the density of $U_t$ (see \cite{LeGallRosenRangeOfStableRandomWalks1991}, $(2.j)$). For example, in the case where $\beta=2$ and where $U$ is the standard Brownian Motion in $\RR^2$ with covariance matrix $\sigma^2\mathrm{Id}$, then we have
\begin{equation*}
    h(n)\underset{n\rightarrow\infty}\sim\frac{\log n}{2\pi\sigma^2}.
\end{equation*}

We are now ready to state our main results, the proofs of which are contained in Section \ref{section:proof of main results}. The first main result is a LLN for $\R_n^k$.
\begin{theorem}
    \label{theorem:LLN}
    Under \ref{Assumption A1} and \ref{Assumption A2}, we have 
    
    Under \ref{Assumption A1}, \ref{Assumption A2}, \ref{Assumption A3}, we have, for all $k\in\{1,\dots,N\}$,
    \begin{equation*}
        \frac{h(n)}{n}\R_n^k\underset{n\rightarrow\infty}{\overset{L^2}\longrightarrow}1.
    \end{equation*}
    Furthermore, if $s(n)\ge1$ for all $n\ge1$, then the previous convergence also holds almost surely.
\end{theorem}
The second is a CLT for $\R_n^k$. As is the case for the range, we show that the fluctuations of the competitive range are non-Gaussian and are described by a random variable that counts the time that the process $U^k$  spends intersecting the past of the trajectories of the $(U^j)_{j\in\{1,\dots,N\}}$ including its own, where $U^\ell$ is the stable process obtained as a scaling limit of $X^\ell$.

\begin{theorem}
Under \ref{Assumption A1}, \ref{Assumption A2}, \ref{Assumption A3}, we have, for all $k\in\{1,\dots,N\}$,
\label{theorem:CLT}
    \begin{equation*}
    \frac{h(n)^2b(n)^d}{n^2}\left(\R_n^k-\EE\left[\R_n^k\right]\right)\distribution -\left(\gamma^k+\sum_{j\ne k} (\alpha^{j,k}-\EE[\alpha^{j,k}])\right)\left([0,1]_\le^2\right),
\end{equation*}
where $U^1,\dots,U^N$ are $i.i.d.$ copies of $U$ and $\alpha^{j,k}$, $\gamma^k$ denote respectively the intersection local time of $U^j,U^k$ and the renormalized self-intersection local time of $U^k$.

Furthermore, the joint convergence in $k\in\{1,\dots,N\}$ also holds.
\end{theorem}
From the viewpoint of our motivations, it is interesting to compare Theorems \ref{theorem:LLN} and \ref{theorem:CLT}. Indeed, we notice that the effect of the competition between individuals is not observed at all at the first order, but rather starts appearing at the scale of the fluctuations.

Before moving on to some preliminary results that will be required in the proof of Theorems \ref{theorem:LLN} and \ref{theorem:CLT}, we will give a brief overview of the strategy of proof and introduce some further notation. For any $M\ge1$, we will write
\begin{equation*}
    [\![M]\!]\eqdef\{1,\dots,M\}.
\end{equation*}
For any non empty disjoint subsets $A,B\subseteq[\![N]\!]$, we define
\begin{equation*}
    \R_n^{A\rightarrow B}\eqdef\sum_{x\in\ZZ^d}\prod_{(a,b)\in A\times B}\1\left(T_x^a\le T_x^b\le n\right).
\end{equation*}
The quantity $\mathcal R_n^{A\rightarrow B}$ corresponds to the number of distinct sites first visited by every single random walk indexed by $A$, then every single random walk indexed by $B$, all before time $n$. If $A=\{j\}$, then we replace the superscript $\{j\}$ by $j$, and do similarly for $B$.

The starting point in the study of $\R_n^k$ is the inclusion-exclusion principle, which allows us to express $\R_n^k$ as a linear combination of $\left|X^k(0,n)\right|$ and of the $\R_n^{S\rightarrow k}$ for non empty $S\subseteq[\![N]\!]\backslash\{k\}$. Then, we establish a limit theorem for the $\R_n^{A\rightarrow B}$ which allow us to notice that for any non empty $S\subseteq[\![N]\!]\backslash\{k\}$, $\R_n^{S\rightarrow k}$ is negligible $a.s.$ in front of $R_n^k$, which yields the LLN. To study the fluctuations of $\R_n^k$, the limit theorem obtained for $\R_n^{A\rightarrow B}$ will tell us that the fluctuations of $\R_n^{S\rightarrow k}$ are of order $\Var(R_n^k)^{1/2}$ if $S$ is a singleton and are negligible in front of $\Var(R_n^k)^{1/2}$ as soon as $|S|\ge2$. 

In order to obtain a limit theorem for $\R_n^{A\rightarrow B}$ we start by showing that, up to some correction involving the truncated Green's function, it is asymptotically equivalent in $L^2$ to some discrete intersection local time $\J_n^{A\rightarrow B}$ involving the walks indexed by $A\cup B$. The second part of the proof then follows the lines of Rosen in Theorem $3$ of \cite{RosenRandomWalks&IntersectionLocalTimes1990}, where the author uses the concepts of shadows and links introduced by Dynkin in \cite{DynkinSelf-IntersectionGauge1988}. This method allows us to show that $\J_n^{A\rightarrow B}$ conveniently scaled converges to its continuous counterpart, which is sufficient to conclude.

\section{Preliminaries}
\label{section:Preliminary Results}

\subsection{Regular variation and miscellaneous results}

We shall make use in what follows of properties of regularly varying functions, and so we recall some basic facts. For more information regarding regularly varying functions, one can for example consult \cite{BinghamGoldieTeugelsRegularVariation1987}.

A function $f:\RR_+\rightarrow\RR_+$ is said to be regularly varying at infinity, or for convenience regularly varying, with index $\kappa\in\RR$ if for all $x\in\RR_+$,
\begin{equation}
\label{eq:definition of regular variation}
    \lim_{t\rightarrow\infty}\frac{f(tx)}{f(t)}=x^\kappa.
\end{equation} 

A regularly varying function with index $0$ is referred to as a slowly varying function. Some examples include, but are not limited to, the functions $t\mapsto(\log t)^p$, $p\in\RR$, or functions that are equivalent to a nonzero constant at infinity.

An immediate consequence of the definition (\ref{eq:definition of regular variation}) is that a regularly varying function $f$ of index $\kappa$ can be written
\begin{equation*}
    f(t)=t^\kappa g(t),
\end{equation*}
where $g$ is a slowly varying function. In particular, since $b$ is regularly varying with index $1/\beta$, throughout this work we shall write $b(n)$ as
\begin{equation*}
    b(n)=n^{1/\beta}s(n), \quad n\ge1,
\end{equation*}
where $s:\RR_+\rightarrow\RR_+$ is a slowly varying function.

Another important property is that the convergence in (\ref{eq:definition of regular variation}) is uniform in $x$ on compacts away from $0$. Indeed, for any compact $K\subseteq\RR_+\backslash\{0\}$, we have 
\begin{equation}
\label{limit:uniform convergence on compacts for regularly varying functions}
    \lim_{t\rightarrow\infty}\sup_{x\in K}\left|\frac{f(tx)}{f(t)}-x^\kappa\right|=0.
\end{equation}
The final property we shall use are the Potter bounds, which state that given a function $f$ which is regularly varying of index $\kappa$, for any $\varepsilon,\delta>0$, there is some constant $C>0$ such that for any $s,t>\delta$,
\begin{equation}
    \label{ineq:Potter bounds}
    C^{-1}\left(\frac{s}{t}\right)^{\kappa-\varepsilon}\wedge\left(\frac{s}{t}\right)^{\kappa+\varepsilon}\le\frac{f(s)}{f(t)}\le C\left(\frac{s}{t}\right)^{\kappa-\varepsilon}\wedge\left(\frac{s}{t}\right)^{\kappa+\varepsilon}.
\end{equation}

We will make use of Skorokhod's version of Donsker's Theorem (\cite{SkorokhodLimitTheoremsStochProc1956}), which states that for all $k\in[\![N]\!]$, $T>0$,

\begin{equation}
    \label{convergence:skorokhod's convergence theorem}
    \left(\frac{X_{\lfloor nt\rfloor}^k}{b(n)}\right)_{t\in[0,T]}\distribution\left(U_t\right)_{t\in[0,T]},
\end{equation}
where the convergence occurs in the usual $J_1$ topology on the space of real-valued càdlàg paths on $[0,T]$. When this result is invoked, we shall refer to the scaled random walks as
\begin{equation*}
    X_t^{k,n}\eqdef\frac{X_{\lfloor nt\rfloor}^k}{b(n)},
\end{equation*}
for all $k\in[\![N]\!]$, $t\ge0$, $n\ge1$.

The following limit will be used.
\begin{lemma}
\label{lemma:limit of h(n)b(n)^d/n}
    \begin{equation*}
        \underset{n\rightarrow\infty}\lim\frac{h(n)b(n)^d}{n}=\infty.
    \end{equation*}
\end{lemma}

\begin{proof}
    If $d>\beta$ then by the Potter bounds, for all $\varepsilon>0$ we may find some $C>0$ such that for $n$ sufficiently large, 
    \begin{equation*}
        \frac{h(n)b(n)^d}{n}\ge Cn^{d/\beta-1-\varepsilon},
    \end{equation*}
    which diverges as soon as we take $\varepsilon$ sufficiently small. If $d=\beta$, then the expression reduces to $h(n)s(n)^d$. Recalling $(2.j)$ from \cite{LeGallRosenRangeOfStableRandomWalks1991}, there is some constant $C>0$ such that 
    \begin{equation*}
        h(n)\underset{n\rightarrow\infty}\sim C\sum_{k=1}^n\frac{1}{ks(k)^d}.
    \end{equation*}
    Furthermore, (\ref{limit:uniform convergence on compacts for regularly varying functions}) implies that for all $p\ge1$, 
    \begin{equation*}
        \underset{n\rightarrow\infty}\lim\max_{n/p\le k\le n}\left|\frac{s(n)^d}{s(k)^d}-1\right|=0,
    \end{equation*}
    whence for any $\varepsilon\in(0,1)$, we may find $n$ sufficiently large so that $s(n)^d/s(k)^d\ge1-\varepsilon$ for all $k\in[\![n/p,n]\!]$. In particular,
    \begin{equation*}
        s(n)^d\sum_{k=1}^n\frac{1}{ks(k)^d}\ge(1-\varepsilon)\sum_{k=n/p}^n\frac1k,
    \end{equation*}
    and taking $n\rightarrow\infty$ yields $\underset{n\rightarrow\infty}\liminf s(n)^d\sum_{k=1}^n\frac{1}{ks(k)^d}\ge (1-\varepsilon)\log p$, and since this is true for any $p\ge1$, we get 
    \begin{equation*}
        \underset{n\rightarrow\infty}\lim s(n)^d\sum_{k=1}^n\frac{1}{ks(k)^d}=\infty,
    \end{equation*}
    which allows us to conclude.
\end{proof}

We shall also need some bounds on the moments of the number of sites in common of $(X^s)_{s\in S}$, where $S\subseteq[\![N]\!]$. To this end, we introduce the notation
    \begin{equation*}
        I_n^S\eqdef\left|\bigcap_{s\in S} X^s(0,n)\right|,
    \end{equation*}
    and the result reads
    \begin{proposition}
        \label{proposition:bounds on moments of intersections}
        for any $S\subseteq[\![N]\!]$, $|S|\ge2$, such that $|S|(d-\beta)<d$ and any $k\ge1$, we have 
        \begin{equation*}
        \EE\left[\left(I_n^S\right)^k\right]\lesssim\left(\frac{n^{|S|}}{h(n)^{|S|}b(n)^{d(|S|-1)}}\right)^k.
        \end{equation*}
    \end{proposition}
    In the case where $|S|=2$, this result is Corollary $3.2.$ of \cite{LeGallRosenRangeOfStableRandomWalks1991}, but for our purpose we need the result for arbitrary $S$.

\begin{proof}
    By the remark following Lemma 3.1 in \cite{LeGallRosenRangeOfStableRandomWalks1991}, we have $\EE\left[(I_n^S)^k\right]\lesssim\EE\left[I_n^S\right]^k$, and so it is sufficient to treat the case $k=1$. Firstly, note that
    \begin{equation*}
        \EE\left[I_n^S\right]=\sum_{x\in\ZZ^d}\PP\left(x\in X^1(0,n)\right)^{|S|}.
    \end{equation*}
    Then, 
    \begin{equation}
    \label{ineq:easy bound for inequality of moments of intersections}
        \sum_{x\in\ZZ^d}\prod_{s\in S}\EE\left[\1(x\in X^s(0,n)\sum_{j=1}^{2n}\1(X_j^s=x)\right]\le \sum_{x\in\ZZ^d}\prod_{s\in S}\EE\left[\sum_{j=1}^{2n}\1\left(X_j^s=x\right)\right],
    \end{equation}
    and for any $s\in S$, applying the Markov property at $T_x^s$ yields
    \begin{align*}
        \EE\left[\1(x\in X^s(0,n)\sum_{j=1}^{2n}\1(X_j^s=x)\right]&=\EE\left[\sum_{\ell=1}^n\1\left(T_x^s=\ell\right)\sum_{j=\ell}^{2n}\PP_x\left(X_{j-\ell}^s=x\right)\right]\\
        &\ge\PP\left(x\in X^s(0,n)\right)\sum_{j=1}^n\PP_x\left(X_j=x\right)=\PP\left(x\in X^s(0,n)\right) h(n).
    \end{align*}
    Therefore, taking the product over $s\in S$ and summing over $x\in\ZZ^d$, since the $X^s$ are $i.i.d.$, we get 
    \begin{equation*}
        h(n)^{|S|}\EE\left[I_n^S\right]\le\sum_{x\in\ZZ^d}\left(\EE\left[\sum_{j=1}^{2n}\1\left(X_j=x\right)\right]\right)^{|S|},
    \end{equation*}
    where $X$ is distributed as $X^1,\dots,X^N$. Now, by introducing the notation $\TT^d\eqdef[-\pi,\pi]^d$ and using the fact that for any $x,y\in\RR^d$, 
    \begin{equation*}
        \1(x=y)=(2\pi)^{-d}\int_{\TT^d}\d z\exp\left(-\i\langle x-y,z\rangle\right),
    \end{equation*}
    we get, by recalling that $\varphi$ is the characteristic function of $X_1$, 
    \begin{align*}
        h(n)^{|S|}\EE\left[I_n^S\right]&\lesssim\sum_{x\in\ZZ^d}\sum_{j_1,\dots,j_{|S|}=1}^{2n}\int_{\TT^{d|S|}}\d z\ \exp\left(\i\sum_{\ell=1}^{|S|}\langle z_\ell, x\rangle\right)\prod_{\ell=1}^{|S|}\varphi(z_\ell)^{j_\ell}\\
        &=\sum_{j_1,\dots,j_{|S|}=1}^{2n}\int_{\TT^{d(|S|-1)}}\d z\left(\prod_{\ell=1}^{|S|-1}\varphi(z_\ell)^{j_\ell}\right)\\
        &\hspace{13em}\times\sum_{x\in\ZZ^d}\exp\left(\i\sum_{\ell=1}^{|S|-1}\langle z_\ell, x\rangle\right)\int_{\TT^d}\d z_{\ell}\ \mathrm e^{\i\langle z_\ell,x\rangle}\varphi(z_\ell)^{j_{|S|}}\\
        &\lesssim\sum_{j_1,\dots,j_{|S|}=1}^{2n}\int_{\TT^{d(|S|-1)}}\d z\left(\prod_{\ell=1}^{|S|-1}\varphi(z_\ell)^{j_\ell}\right)\sum_{x\in\ZZ^d}\exp\left(\i\sum_{\ell=1}^{|S|-1}\langle z_\ell,x\rangle\right)\widehat{\varphi^{j_{|S|}}}(x)\\
        &\lesssim\sum_{j_1,\dots,j_{|S|}=1}^{2n}\int_{\TT^{d(|S|-1)}}\d z\ \varphi\left(-\sum_{\ell=1}^{|S|-1}z_\ell\right)^{j_{|S|}}\prod_{\ell=1}^{|S|-1}\varphi(z_\ell)^{j_\ell},
    \end{align*}
    where we used Fourier inversion in the third line and Fourier transform in the fourth. Letting $\TT_n^d\eqdef[-\pi b(n),\pi b(n)]^d$, we get
    \begin{equation}
    \label{ineq:last bound for moments of intersections}
        \frac{h(n)^{|S|}b(n)^{d(|S|-1)}}{n^{|S|}}\EE\left[I_n^S\right]\lesssim\frac{1}{n^{|S|}}\sum_{j_1,\dots,j_{|S|}=1}^{2n}\int_{\TT_n^{d(|S|-1)}}\d z\ \varphi\left(-\sum_{\ell=1}^{|S|-1}\frac{z_\ell}{b(n)}\right)^{j_{|S|}}\prod_{\ell=1}^{|S|-1}\varphi\left(\frac{z_\ell}{b(n)}\right)^{j_\ell}.
    \end{equation}
    According to $(5.15)$ in \cite{RosenRandomWalks&IntersectionLocalTimes1990}, for any $\varepsilon>0$ sufficiently small and for all $x\in\RR^d$, $n\ge1$,
    \begin{equation}
    \label{ineq:riemann sum of scaled characteristic is bounded by a power of x}
        \frac{1}{n}\sum_{k=0}^n\left|\varphi\left(\frac{x}{b(n)}\right)\right|^k\lesssim\ \frac{1}{1+\overline x^{\beta(1-\varepsilon)}}=:f(\overline x),
    \end{equation}
    where $\overline{x}\equiv x\pmod{\pi b(n)} $. In particular, by writing $z_{|S|}\eqdef-\sum_{\ell=1}^{|S|-1}z_\ell$ and defining for $z\in\RR^{d(|S|-1)}$ and $m\in[\![|S|-1]\!]$ the functions
    \begin{equation*}
        F_m(z)\eqdef f(z_{|S|})\prod_{\substack{\ell=1 \\ \ell\ne m}}^{|S|-1}f(z_\ell),\quad  
        F_{|S|}(z)\eqdef\prod_{\ell=1}^{|S|-1}f(z_\ell),
    \end{equation*}
    then using \eqref{ineq:riemann sum of scaled characteristic is bounded by a power of x}, the right hand side of \eqref{ineq:last bound for moments of intersections} is bounded up to a constant by 
    \begin{equation*}
        \int_{\TT_n^{d(|S|-1)}}\d z\prod_{m=1}^{|S|}F_m(\overline z)^{(|S|-1)^{-1}}\lesssim \prod_{m=1}^{|S|}\left(\int_{\RR^{d(|S|-1)}}\d z \ F_m(z)^{\frac{|S|}{|S|-1}}\right)^{1/|S|},
    \end{equation*}
    by Hölder's inequality.
    Examining the integral coordinate by coordinate, it is clear that the term $m=|S|$ in the previous expression is finite so long as $\frac{|S|}{|S|-1}(\beta-\varepsilon)>d$. Thanks to our assumption, we may choose $\varepsilon$ sufficiently small so that this is indeed the case. If $m\ne|S|$, then by letting $\widetilde f(x)\eqdef f(x)^{\frac{|S|}{|S|-1}}$ for $x\in\RR^d$ and $\widetilde F_m$ similarly, then
    \begin{equation*}
        \int_{\RR^{d(|S|-1)}}\d z\ \widetilde F_m(z)=\int_{\RR^{d(|S|-2)}}\d z_1\dots\d z_{m-1}\d z_{m+1}\dots\d z_{|S|-1}\left(\prod_{\substack{\ell=1 \\ \ell\ne m}}^{|S|-1}\widetilde f(z_\ell)\right)\int_{\RR^d}\d z_m\ \widetilde f(z_{|S|}).
    \end{equation*}
    The innermost integral is simply a power of the $L^p$ norm of $f$ with $p=\frac{|S|}{|S|-1}$. As before, this quantity is bounded if $\frac{|S|}{|S|-1}(\beta-\varepsilon)>d$, and then so is the rest of the integral, which is sufficient to conclude.
\end{proof}

\begin{remark}
    As proved in \cite{RosenRandomWalks&IntersectionLocalTimes1990}, inequality \eqref{ineq:riemann sum of scaled characteristic is bounded by a power of x} is true only when $\beta>1$. Indeed, the proof relies on showing that a certain inequality (noted (5.8) in \cite{RosenRandomWalks&IntersectionLocalTimes1990}) over the intervals $I(n)\eqdef[1/2b(n),1/b(n)]$ for $n\ge n_0$ where $n_0$ is taken sufficiently large, and then showing that the intervals $I(2^kn_0)$, $k\ge0$, overlap two-by-two when $\beta>1$, in such a way that the inequality actually holds over a neighbourhood of the origin. The intervals overlap two-by-two when $n_0$ is sufficiently large because of the fact that that $b$ is regularly varying of index $1/\beta$, which yields
    \begin{equation*}
        \frac{2b(n)}{b(2n)}\underset{n\rightarrow\infty}\longrightarrow \frac{2}{2^{1/\beta}}>1,
    \end{equation*}
    and so $1/b(2n)>1/2b(n)$ for sufficiently large $n$.

    However, \eqref{ineq:riemann sum of scaled characteristic is bounded by a power of x} still holds for any $\beta>0$. Indeed, for any $\alpha\in(1,2)$, one can consider the intervals $I(\lfloor\alpha^k\rfloor n_0)$, $k\ge0$, which will overlap two-by-two when $n_0$ is taken sufficiently large so long as 
    \begin{equation*}
        \frac{2}{\alpha^{1/\beta}}>1\iff \beta>\frac{\log\alpha}{\log 2}.
    \end{equation*}
    This proves the claim since we can take $\alpha$ arbitrarily close to $1$.
\end{remark}

We end this section with a bound involving the Fourier transform of a probability density function in the Schwartz space $\mathcal S(\RR^d)$ which will be useful in Section \ref{section:Limit Theorems for R^(A->B) and I^(A->B)}.

\begin{lemma}
    \label{ineq:bound for products of Fourier transform}
    Let $f\in\mathcal S(\RR^d)$ and suppose that $\int_{\RR^d}\d x f(x)=1$. We let $\hat f$ denote the Fourier transform of $f$. Then for any $n\ge1$, $\delta\in(0,1)$, we may find a constant $C$ such that for any $x_1,\dots,x_n\in\RR^d$,
    \begin{equation*}
        \left|1-\prod_{k=1}^n\hat f(x_k)\right|\le C\sum_{k=1}^n|x_k|^\delta.
    \end{equation*}
\end{lemma}

\begin{proof}
    Since $\int_{\RR^d}f=1$, $\hat f(0)=1$ and since $f\in\mathcal S(\RR^d)$, we also have $\hat f\in\mathcal S(\RR^d)$. In particular for $|x|\le 1$, the mean-value theorem yields
    \begin{equation*}
        \left|1-\hat f(x)\right|\le\lVert\nabla \hat f\rVert_\infty|x|\le \lVert\nabla \hat f\rVert_\infty|x|^\delta.
    \end{equation*}
    Noting that for $|x|>1$, we can simply write
    \begin{equation*}
        \left|1-\hat f(x)\right|\le 1+\lVert \hat f\rVert_\infty\le (1+\lVert \hat f\rVert_\infty)|x|^\delta,
    \end{equation*}
    whence $|1-\hat f(x)|\le C|x|^\delta$  for some $C>0$ uniformly for $x\in\RR^d$. Letting $x_0=0$, we now have 
    \begin{align*}
        \left|1-\prod_{k=1}^n\hat f(x_k)\right|&=\left|\sum_{k=1}^n \left(1-\hat f(x_k)\right)\prod_{j=0}^{k-1}\hat f(x_j)\right|\\
        &\le \sum_{k=1}^n\lVert \hat f\rVert_\infty^k\left|1-\hat f(x_k)\right|\\
        &\le C\lVert\hat f\rVert_\infty^n\sum_{k=1}^n|x_k|,
    \end{align*}
    since $\lVert\hat f\rVert_\infty\ge1$, which concludes.
\end{proof}

\subsection{Intersection Local Times of Stable Processes}

If one considers $N$ independent $\beta$-stable Lévy processes $U^1,\dots,U^N$ in $\RR^d$, then it was shown by Taylor (\cite{TaylorMultPointsStable1966}) that their paths $a.s.$ share common points if and only if $N(d-\beta)<d$. When this is the case, one can construct a non-trivial random measure $\alpha^{[\![N]\!]}$ on $\RR_+^N$ defined by the formal expression
\begin{equation*}
    \alpha^{[\![N]\!]}(A)\eqdef\int_A\d t_1\dots\d t_k\prod_{k=1}^{N-1}\delta_0\left(U_{t_{k+1}}^{k+1}-U_{t_k}^k\right),
\end{equation*}
where $A$ is any Borel set in $\RR_+^N$. The measure $\alpha^{[\![N]\!]}$ is supported on the random set 
\begin{equation*}
    \left\{(t_1,\dots,t_N)\in\RR_+^N, U_{t_1}^1=\dots=U_{t_N}^N\right\},
\end{equation*}
and is referred to as the intersection local time of the processes $U^1,\dots, U^N$. The measure $\alpha^{[\![N]\!]}$ is non-trivial in the sense that, for example, $\alpha^{[\![N]\!]}([0,1]^N)\in(0,\infty)$ $a.s.$. One way to construct $\alpha^{[\![N]\!]}$ is to define a smoothed version $\alpha_\varepsilon^{[\![N]\!]}$ by replacing $\delta_0$ by a mollifier $f_\varepsilon$, and then to show that letting $\varepsilon\rightarrow 0$ yields an $a.s.$ limit (see for instance \cite{ChenRosen2005}, Theorem $9$). Once we know $\alpha^{[\![N]\!]}$ is well-defined, then one can also define $\alpha^S$, for any $S\eqdef\{s_1,\dots,s_{|S|}\}\subseteq[\![N]\!]$ with $|S|\ge2$, as
\begin{equation*}
    \alpha^S(A)\eqdef\int_A\d t_1\dots\d t_{|S|}\prod_{k=1}^{|S|-1}\delta_0\left(U_{t_{k+1}}^{s_{k+1}}-U_{t_k}^{s_k}\right), \quad A\subseteq\RR_+^N.
\end{equation*}
If $S\eqdef\{j,k\}$, then we shall write $\alpha^{j,k}$ instead of $\alpha^{\{j,k\}}$.

A problem that is dual to studying intersection of $N$ independent $\beta$-stable processes is to study the $N$-fold self-intersections of a single $\beta$-stable process $U$. Focusing on the case $N=2$, in the same spirit as what we previously saw, one is tempted to define the self-intersection local time of $U^k$ as the random measure
\begin{equation*}
    \alpha^k\left(A\right)\eqdef\int_A\d s \d t\ \delta_0\left(U_t^k-U_s^k\right),\quad A\subseteq\left(\RR_+^2\right)_\le\eqdef\{(s,t)\in\RR_+^2, s\le t\}.
\end{equation*}
By the Markov property, $\alpha^k$ is well defined on sets of the form $[a,b)\times(c,d]$ with $b\le c$ thanks to the equality in distribution
\begin{equation*}
    \alpha^k([a,b)\times(c,d])\overset{\mathrm{(d)}}=\widetilde \alpha([a,b)\times(c,d]),
\end{equation*}
where $U,\widetilde U$ are independent and distributed as $U^k$ and $\widetilde\alpha$ denotes the intersection local time of $U,\widetilde U$. In particular, by letting 
\begin{equation*}
    A_{i,j}\eqdef\left[\frac{2i-2}{2^j},\frac{2i-1}{2^j}\right)\times\left(\frac{2i-1}{2^j},\frac{2i}{2^j}\right],\quad j\ge1,i\in[\![2^{j-1}]\!],
\end{equation*}
Then $\alpha^k(A_{i,j})<\infty$ $a.s.$ for all admissible $(i,j)$. However, one can show through a scaling argument that
\begin{equation*}
    \sum_{j\ge1}\sum_{i=1}^{2^{j-1}}\EE\left[\alpha^k(A_{i,j})\right]=\infty,\quad  \sum_{j\ge1}\sum_{i=1}^{2^{j-1}}\mathrm{Var}\left(\alpha^k\left(A_{i,j}\right)\right)<\infty.
\end{equation*}
As a direct consequence, $\alpha^k\left([0,1]_\le^2\right)=\infty$ $a.s.$, but
\begin{equation*}
    \gamma^k\left([0,1]_\le^2\right)\eqdef\sum_{j\ge1}\sum_{i=1}^{2^{j-1}}\left(\alpha^k(A_{i,j})-\EE\left[\alpha^k(A_{i,j})\right]\right)
\end{equation*}
converges in $L^2$. The quantity $\gamma^k\left([0,1]_\le^2\right)$ is referred to as the renormalized self-intersection local time of $U^k$. For more information regarding self-intersection local times of stable processes, one may consult \cite{BassChenRosenLDSILT}, \cite{ChenRosen2005}.

\section{Asymptotic behaviour of ordered intersections}
\label{section:Limit Theorems for R^(A->B) and I^(A->B)}

The aim of this section is to establish limit theorems for the competitive ranges $\R_n^{A\rightarrow B}$. We will see that up to some scaling, both random variables converge in distribution to some intersection local time involving the limiting $\beta$-stable processes. To start off, we recall the following result from \cite{LeGallRosenRangeOfStableRandomWalks1991}. Here, for any $t>0$, $p_t$ denotes the density of the random variable $U_t$. 

\begin{lemma}[Lemma $6.2.$, \cite{LeGallRosenRangeOfStableRandomWalks1991}]
\label{lemma:uniform convergence of hitting times}
    Suppose that $d-1<\beta\le d$. Then for all $\ell\in[\![N]\!]$ and each $x,y\in\RR^d$ with $x\neq y$,
    \begin{equation}
    \label{limit:uniform convergence discrete -> continuous I}
        \underset{n\rightarrow\infty}\lim\ \underset{0\le k_1\le k_2\le n}\sup\left|\frac{h(n)b(n)^d}{n}\PP_{\lfloor b(n) x\rfloor}\left(k_1\le T_{\lfloor b(n) y\rfloor}^\ell\le k_2\right)-\int_{k_1/n}^{k_2/n}\d s\ p_s(y-x)\right|=0,
    \end{equation}
    \begin{equation}
    \label{limit:uniform convergence discrete -> continuous II}
        \underset{n\rightarrow\infty}\lim\ \underset{0\le k_1\le k_2\le n}\sup\left|\frac{b(n)^d}{n}\sum_{j=k_1}^{k_2}\PP_{\lfloor b(n)x\rfloor}\left(X_j^\ell=\lfloor b(n) y\rfloor\right)-\int_{k_1/n}^{k_2/n}\d s\ p_s(y-x)\right|=0.
    \end{equation}
    Furthermore, for any $\varepsilon>0$ and $x,y\in\RR^d$,
    \begin{equation}
    \label{ineq:bound for dominated convergence theorem}
       \frac{h(n)b(n)^d}{n}\PP_{\lfloor b(n)x\rfloor}\left(T_{\lfloor b(n) y\rfloor}^\ell\le n\right)\quad\text{and}\quad\frac{b(n)^d}{n}\sum_{k=0}^n\PP_{\lfloor b(n)x\rfloor}\left(X_k^\ell=\lfloor b(n) y\rfloor\right)
    \end{equation}
    are both $\lesssim|x-y|^{\beta-d-\varepsilon}+|x-y|^{\beta-d+\varepsilon}$ uniformly in $n$.
\end{lemma}

We may now move to the main results of this section. Define for all $d\ge1, \sigma\in\mathfrak S([\![d]\!])$,
\begin{equation*}
    \Delta_\sigma\eqdef\left\{(s_1,\dots, s_d)\in[0,1]^d\mid s_{\sigma(1)}\le\dots\le s_{\sigma(d)}\right\}.
\end{equation*}
Furthermore, we define for $A,B\subseteq[\![1,N]\!]$ disjoint
\begin{equation*}
    \mathfrak S^{A\rightarrow B}\eqdef\left\{\sigma\in\mathfrak S(A\cup B)\mid \forall (a,b)\in A\times B,\sigma(a)\le\sigma(b)\right\},
\end{equation*}
and
\begin{equation*}
    \Delta^{A\rightarrow B}\eqdef\bigcup_{\sigma\in\mathfrak S^{A\rightarrow B}}\Delta_\sigma.
\end{equation*}
We also introduce the discrete analogues
\begin{equation*}
    \Delta_{n,\sigma}\eqdef\left\{(k_1,\dots, k_d)\in[\![n]\!]^d\mid k_{\sigma(1)}\le\dots\le k_{\sigma(d)}\right\},
\end{equation*}
and
\begin{equation*}
    \Delta_n^{A\rightarrow B}\eqdef\bigcup_{\sigma\in\mathfrak S^{A\rightarrow B}}\Delta_{n,\sigma}.
\end{equation*}
We recall that we work under assumptions \ref{Assumption A1}, \ref{Assumption A2}, \ref{Assumption A3}.

\begin{theorem}
\label{theorem:convergence of R_n^(A->B)}
    So long as $|A\cup B|(d-\beta)<d$, we have the following convergence in distribution
    \begin{equation}
        \label{convergence:limit of R_n^(A->B)}
        \frac{h(n)^{|A\cup B|}b(n)^{d(|A\cup B|-1)}}{n^{|A\cup B|}}\mathcal R_n^{A\rightarrow B}\distribution \alpha^{A\cup B}\left(\Delta^{A\rightarrow B}\right),
    \end{equation}
\end{theorem}

We follow the strategy of proof used in \cite{LeGallRosenRangeOfStableRandomWalks1991}. We start by introducing a discrete intersection local time $\mathcal J_n^{A\rightarrow B}$ for which we can directly study the convergence in distribution under proper scaling. We shall then show that, up to some scaling, this discrete intersection local time is asymptotically equivalent to $\mathcal R_n^{A\rightarrow B}$ in a certain strong sense, which is enough to conclude. We start by introducing the quantity $\mathcal J_n^{A\rightarrow B}$. Supposing without loss of generality that $A\eqdef[\![|A|]\!]$, $B=|A|+[\![|B|]\!]$, we let
\begin{equation*}
    \mathcal J_n^{A\rightarrow B}\eqdef\sum_{x\in\ZZ^d}\sum_{k\in \Delta_n^{A\rightarrow B}}\prod_{\ell\in A\cup B}\1\left(X_{k_\ell}^{\ell}=x\right).
\end{equation*}
The process $\mathcal J_n^{A\rightarrow B}$ is a discrete intersection local time that counts the time that the walks $(X^b)_{b\in B}$ spend intersecting the past of the walks $(X^a)_{a\in A}$. Firstly, we will establish that, up to scaling, $\R_n^{A\rightarrow B}$ and $\J_n^{A\rightarrow B}$ are asymptotically equivalent in $L^2$.

\subsection{From ordered intersections to discrete intersection local time}

\begin{proposition}
\label{proposition: h(n)^(a+b)R_n^(A->B) and J_n^(A->B) are close in L^2}
    Suppose that $|A\cup B|(d-\beta)<d$. Then
    \begin{equation*}
        \frac{b(n)^{d(|A\cup B|-1)}}{n^{|A\cup B|}}\left(h(n)^{|A\cup B|}\mathcal R_n^{A\rightarrow B}-\mathcal J_n^{A\rightarrow B}\right)\underset{n\rightarrow\infty}{\overset{L^2}\longrightarrow}0.
    \end{equation*}
\end{proposition}

\begin{proof}
    We start by taking $K>0$ and considering the following versions of $\mathcal R_n^{A\rightarrow B}$ and $\mathcal J_n^{A\rightarrow B}$ with a cutoff,
    \begin{equation*}
        \mathcal R_{n,K}^{A\rightarrow B}\eqdef\sum_{\substack{x\in\ZZ^d \\ |x|\le b(n)K}}\prod_{(a,b)\in A\times B}\1\left(T_x^a\le T_x^b\le n\right),
    \end{equation*}
    \begin{align*}
        \mathcal J_{n,K}^{A\rightarrow B}&\eqdef\sum_{\substack{x\in\ZZ^d \\ |x|\le b(n) K}}\sum_{k\in\Delta_n^{A\rightarrow B}}\prod_{\ell\in A\cup B}\1\left(X_{k_\ell}^{\ell}=x\right).
    \end{align*}
    As in the proof of Proposition $6.3.$ of \cite{LeGallRosenRangeOfStableRandomWalks1991}, it is sufficient by (\ref{convergence:skorokhod's convergence theorem}) to show that the result holds for the versions with a cutoff. We have
    \begin{align*}
        &\EE\left[\left(h(n)^{|A\cup B|}\mathcal R_{n,K}^{A\rightarrow B}-\mathcal J_{n,K}^{A\rightarrow B}\right)^2\right]\\
        &\hspace{4em}=h(n)^{2|A\cup B|}\EE\left[\left(\mathcal R_{n,K}^{A\rightarrow B}\right)^2\right]-2h(n)^{|A\cup B|}\EE\left[\mathcal R_{n,K}^{A\rightarrow B}\mathcal J_{n,K}^{A\rightarrow B}\right]+\EE\left[\left(\mathcal J_{n,K}^{A\rightarrow B}\right)^2\right]\\
        &\hspace{4em}=:I_n^1-2I_n^2+I_n^3.
    \end{align*}
    We let $\mathcal B_{K}^{(n)}\eqdef\left\{x\in\RR^d, |\lfloor b(n) x\rfloor|\le K b(n)\right\}$ and $\mathcal B_K$ denote the Euclidean ball of radius $K$. Then
    \begin{align}
        &\frac{b(n)^{2d(|A\cup B|-1)}}{n^{2|A\cup B|}}I_n^1\nonumber\\
        &\hspace{1em}=\frac{h(n)^{2|A\cup B|}b(n)^{2d(|A\cup B|-1)}}{n^{2|A\cup B|}}\sum_{\substack{x,y\in\ZZ^d \\ |x|,|y|\le b(n) K}}\PP\left(\forall(a,b)\in A\times B,T_x^{a}\le T_x^{b}\le n, T_y^{a}\le T_y^{b}\le n\right)\nonumber\\
        &\hspace{1em}=\int_{\mathcal B_K^{(n)}\times\mathcal B_K^{(n)}}\d x\d y\frac{h(n)^{2|A\cup B|}b(n)^{2d|A\cup B|}}{n^{2|A\cup B|}}\sum_{k\in\widetilde\Delta_n^{A\rightarrow B}}\prod_{\ell\in A\cup B}\PP\left(T_{\lfloor b(n)x\rfloor}^{\ell}=k^1_{\ell},T_{\lfloor b(n)y\rfloor}^{\ell}=k_\ell^2\right),\label{eq:modified expression of I_n^1}
    \end{align}
    where we let 
    \begin{align*}
        \widetilde \Delta_n^{A\rightarrow B}\eqdef&\left\{k=(k_\ell^1,k_\ell^2)_{\ell\in A\cup B}\in[\![n]\!]^{2|A\cup B|}\mid k^i_{a}\le k^i_{{b}}, (a,b)\in A\times B, i=1,2\right\}.
    \end{align*}

    We start by studying the pointwise limit of the integrand in (\ref{eq:modified expression of I_n^1}). Note that $\widetilde\Delta_n^{A\rightarrow B}$ can be partitioned into $2^{|A\cup B|}$ disjoint subsets depending on whether $k_\ell^1\le k_\ell^2$ or $k_\ell^2<k_\ell^1$ for each $\ell\in A\cup B$. For any $S\subseteq A\cup B$, we let $\widetilde \Delta_{n,S}^{A\rightarrow B}$ denote the elements $k\in\widetilde\Delta_n^{A\rightarrow B}$ such that $k_\ell^1\le k_\ell^2\iff \ell\in S$. For the sake of simplicity, we shall work on $\widetilde\Delta_{n,A\cup B}^{A\rightarrow B}$ but the other cases are treated identically. In what follows, we let $x_n\eqdef\lfloor b(n) x\rfloor$ for any $x\in\RR^d$, and so for any $k\in\widetilde\Delta_{n,A\cup B}^{A\rightarrow B}$, $\ell\in A\cup B$, the Markov property yields
    \begin{align*}
        \PP\left(T_{x_n}^\ell=k_\ell^1,T_{y_n}^\ell=k_\ell^2\right)&=\PP\left(T_{x_n}^\ell=k_\ell^1\right)\PP_{x_n}\left(T_{y_n}^\ell=k_\ell^2-k_\ell^1\right)\\
        &\hspace{5em}-\PP\left(T_{y_n}^\ell<T_{x_n}^\ell=k^1_\ell \right)\PP_{x_n}(T_{y_n}^\ell=k_\ell^2-k_\ell^1)
    \end{align*}
    and so
    \begin{align*}
        &\sum_{k\in\widetilde\Delta_{n,A\cup B}^{A\rightarrow B}}\prod_{\ell\in A\cup B}\PP\left(T_{\lfloor b(n)x\rfloor}^{\ell}=k^1_{\ell},T_{\lfloor b(n)y\rfloor}^{\ell}=k_\ell^2\right)\\
        &\hspace{5em}=\sum_{L\subseteq A\cup B}(-1)^{|L|}\sum_{k\in\widetilde\Delta_{n,A\cup B}^{A\rightarrow B}}\left(\prod_{\ell\in L}\PP\left(T_{x_n}^\ell=k_1^\ell\right)\PP_{x_n}\left(T_{y_n}^\ell=k_\ell^2-k_\ell^1\right)\right)\\
        &\hspace{10em}\times\left(\prod_{\ell\notin L}\PP\left(T_{y_n}^\ell<T_{x_n}^\ell=k^1_\ell \right)\PP_{x_n}(T_{y_n}^\ell=k_\ell^2-k_\ell^1)\right)\\
        &=:\sum_{L\subseteq A\cup B}\Sigma_{n,L}(x,y)
    \end{align*}
    According to \eqref{ineq:bound for dominated convergence theorem}, we have for any $L\subseteq A\cup B$, the bound
    \begin{equation*}
        \Sigma_{n,L}(x,y)\lesssim\left(\prod_{\ell\in L}\frac{n^2}{h(n)^2b(n)^{2d}}\right)\left(\prod_{\ell\notin L}\frac{n^3}{h(n)^3b(n)^{3d}}\right)=\left(\frac{n}{h(n)b(n)^d}\right)^{3|A\cup B|-|S|},
    \end{equation*}
    where the constant depends on $x,y$. In particular, as soon as $L\ne A\cup B$,
    \begin{equation*}
        \frac{h(n)^{2|A\cup B|}b(n)^{2d|A\cup B|}}{n^{2|A\cup B|}}\Sigma_{n,L}(x,y)\lesssim \frac{n}{h(n)b(n)^d}=o(1),
    \end{equation*}
    thanks to Lemma \ref{lemma:limit of h(n)b(n)^d/n}. It therefore remains to study the case $L=A\cup B$. For this, we define the continuous analogues
    \begin{equation*}
        \widetilde\Delta^{A\rightarrow B}\eqdef\left\{s=(s_\ell^1,s_\ell^2)_{\ell\in A\cup B}\in[0,1]^{2|A\cup B|}\mid s_a^i\le s_b^i, (a,b)\in A\times B,i=1,2\right\},
    \end{equation*}
    and
    \begin{equation*}
        \widetilde\Delta_{L}^{A\rightarrow B}\eqdef\left\{s\in\widetilde\Delta^{A\rightarrow B}\mid s_\ell^1\le s_\ell^2\text{ iff }\ell\in L\right\}.
    \end{equation*}
    Then by \eqref{limit:uniform convergence discrete -> continuous I}, we get
    \begin{align*}
        &\frac{h(n)^{2|A\cup B|}b(n)^{2d|A\cup B|}}{n^{2|A\cup B|}}\Sigma_{n,A\cup B}(x,y)\\
        &\hspace{5em}=\sum_{k\in\widetilde\Delta_{n,A\cup B}^{A\rightarrow B}}\prod_{\ell\in A\cup B}\left(\frac{h(n)b(n)^d}{n}\PP\left(T_{x_n}^\ell=k_\ell^1\right)\right)\left(\frac{h(n)b(n)^d}{n}\PP_{x_n}\left(T_{y_n}^\ell=k_\ell^2-k_\ell^1\right)\right)\\
        &\hspace{4em}\underset{n\rightarrow\infty}{\longrightarrow}\int_{\widetilde \Delta_{A\cup B}^{A\rightarrow B}}\d s\ \prod_{\ell\in A\cup B}p_{s_\ell^1}(x)p_{s_\ell^2-s_\ell^1}(y-x).
    \end{align*}
    Reasoning identically over each $\widetilde \Delta_{n,S}^{A\rightarrow B}$, we get that the pointwise limit of the integrand in \eqref{eq:modified expression of I_n^1} is 
    \begin{equation*}
        \sum_{S\subseteq A\cup B}\int_{\widetilde\Delta_S^{A\rightarrow B}}\d s\left(\prod_{\ell\in S}p_{s_\ell^1}(x) p_{s_\ell^2-s_\ell^1}(y-x)\right)\left(\prod_{\ell\notin S}p_{s_\ell^2}(y) p_{s_\ell^1-s_\ell^2}(x-y)\right).
    \end{equation*}
    It remains to justify that we may apply the dominated convergence theorem. For any $x,y\in\mathcal B_K$, $x\ne y$, we have
    \begin{align*}
        &\frac{h(n)^{2|A\cup B|}b(n)^{2d|A\cup B|}}{n^{2|A\cup B|}}\PP\left(\forall(a,b)\in A\times B, T_{\lfloor b(n)x\rfloor}^a\le T_{\lfloor b(n)x\rfloor}^b\le n,T_{\lfloor b(n)y\rfloor}^a\le T_{\lfloor b(n)y\rfloor}^b\le n\right)\\
        &\hspace{3em}\lesssim \frac{h(n)^{2|A\cup B|}b(n)^{2d|A\cup B|}}{n^{2|A\cup B|}}\PP\left(\forall\ell \in A\cup B, T_{\lfloor b(n) x\rfloor}^{\ell},T_{\lfloor b(n) y\rfloor}^{\ell}\le n\right)\\
        &\hspace{3em}=\prod_{\ell\in A\cup B}\left(\frac{h(n)^2b(n)^{2d}}{n^2}\PP\left(T_{\lfloor b(n) x\rfloor}^{{\ell}},T_{\lfloor b(n) y\rfloor}^{\ell}\le n\right)\right)\\
        &\hspace{3em}\overset{\eqref{ineq:bound for dominated convergence theorem}}{\lesssim}\left(\left(|x|^{\beta-d-\varepsilon}+|x|^{\beta-d+\varepsilon}\right)\left(|y-x|^{\beta-d-\varepsilon}+|y-x|^{\beta-d+\varepsilon}\right)\right.\\
        &\hspace{6em}\left.+\left(|y|^{\beta-d-\varepsilon}+|y|^{\beta-d+\varepsilon}\right)\left(|x-y|^{\beta-d-\varepsilon}+|x-y|^{\beta-d+\varepsilon}\right)\right)^{|A\cup B|}\\
        &\hspace{3em}\lesssim\left(\left(|x|^{\beta-d-\varepsilon}+|x|^{\beta-d+\varepsilon}\right)^{|A\cup B|}+\left(|y|^{\beta-d-\varepsilon}+|y|^{\beta-d+\varepsilon}\right)^{|A\cup B|}\right)\\
        &\hspace{6em}\times\left(|x-y|^{\beta-d-\varepsilon}+|x-y|^{\beta-d+\varepsilon}\right)^{|A\cup B|}
    \end{align*}
    Note that the uniform bound we have obtained is integrable on $\mathcal B_K\times\mathcal B_K$ as soon as $|A\cup B|(\beta-d+\varepsilon)<d$, and by hypothesis we may choose $\varepsilon$ sufficiently small so that this condition is satisfied. By the dominated convergence theorem, we have thus shown that
    \begin{align*}
        &\lim_{n\rightarrow\infty}\frac{b(n)^{2d(|A\cup B|-1)}}{n^{2|A\cup B|}}I_n^1\\
        &\hspace{4em}=\sum_{S\subseteq A\cup B}\int_{\mathcal B_K\times\mathcal B_K}\d x\d y\int_{\widetilde\Delta_S^{A\rightarrow B}}\d s\left(\prod_{\ell\in S}p_{s_\ell^1}(x) p_{s_\ell^2-s_\ell^1}(y-x)\right)\left(\prod_{\ell\notin S}p_{s_\ell^2}(y) p_{s_\ell^1-s_\ell^2}(x-y)\right)
    \end{align*}
    We now turn to the study of $I_n^3$, which is essentially identical to that of $I_n^1$. The same calculation as before yields
    \begin{align*}
        \frac{b(n)^{d(|A\cup B|-1)}}{n^{2|A\cup B|}}I_n^3&=\int_{\mathcal B_K^{(n)}\times\mathcal B_K^{(n)}}\d x\d y\sum_{k\in\widetilde\Delta_n^{A\rightarrow B}}\prod_{\ell\in A\cup B}\left(\frac{b(n)^{2d}}{n^2}\PP\left(X_{k_\ell^1}^{\ell}=\lfloor b(n)x\rfloor,X_{k_\ell^2}^{\ell}=\lfloor b(n)y\rfloor\right)\right),
    \end{align*}
    the justification of the dominated convergence theorem is identical to the one provided for $I_n^1$ and we show that the integrand converges to the same one as in the case of $I_n^1$ using (\ref{limit:uniform convergence discrete -> continuous II}) instead of (\ref{limit:uniform convergence discrete -> continuous I}). All in all, we get
    \begin{align*}
        &\lim_{n\rightarrow\infty}\frac{b(n)^{2d(|A\cup B|-1)}}{n^{2|A\cup B|}}I_n^3\\
        &\hspace{4em}=\sum_{S\subseteq A\cup B}\int_{\mathcal B_K\times\mathcal B_K}\d x\d y\int_{\widetilde\Delta_S^{A\rightarrow B}}\d s\left(\prod_{\ell\in S}p_{s_\ell^1}(x) p_{s_\ell^2-s_\ell^1}(y-x)\right)\left(\prod_{\ell\notin S}p_{s_\ell^2}(y) p_{s_\ell^1-s_\ell^2}(x-y)\right)
    \end{align*}

    The method for studying $I_n^2$ is a combination of the methods used for studying $I_n^1$ and $I_n^3$. We invoke (\ref{limit:uniform convergence discrete -> continuous I}) and (\ref{limit:uniform convergence discrete -> continuous II}) to justify the convergence for any fixed $x,y\in\RR^d$ and to show that some remainder term analogous to the one obtained in the study of $I_n^1$ converges to $0$. We then use (\ref{ineq:bound for dominated convergence theorem}) to justify the dominated convergence theorem, and in doing so we show that 
    \begin{align*}
        &\lim_{n\rightarrow\infty}\frac{b(n)^{2d(|A\cup B|-1)}}{n^{2|A\cup B|}}I_n^2\\
        &\hspace{4em}=\sum_{S\subseteq A\cup B}\int_{\mathcal B_K\times\mathcal B_K}\d x\d y\int_{\widetilde\Delta_S^{A\rightarrow B}}\d s\left(\prod_{\ell\in S}p_{s_\ell^1}(x) p_{s_\ell^2-s_\ell^1}(y-x)\right)\left(\prod_{\ell\notin S}p_{s_\ell^2}(y) p_{s_\ell^1-s_\ell^2}(x-y)\right)
    \end{align*}
     This is enough to conclude, since the quantity we are interested in is equal to
     \begin{equation*}
         \frac{b(n)^{2d(|A\cup B|-1)}}{n^{2|A\cup B|}}\left(I_n^1-2I_n^2+I_n^3\right).
     \end{equation*}

\end{proof}

\subsection{Scaling limit of the discrete intersection local time}

We now turn to the study of $\J_n^{A\rightarrow B}$. As mentioned in Section \ref{susbection:main results, strategy of proof}, we use the concept of links and shadows introduced by Dynkin \cite{DynkinSelf-IntersectionGauge1988}. Given any functional $F$ which is continuous in the product $J_1$ topology, we have
\begin{equation*}
    F\left(\left(X^{s,n}\right)_{s\in A\cup B}\right)\distribution F((U^s)_{s\in A\cup B}).
\end{equation*}
In practice though, we are led to study discontinuous functionals. We are not helped by the fact that the continuous analogue $\alpha^{A\cup B}(\Delta^{A\rightarrow B})$ is also a discontinuous functional of the trajectories. In order to circumvent this issue, we introduce for any $\varepsilon>0$ a link $\J_{n,\varepsilon}^{A\rightarrow B}$ and a shadow $\alpha_\varepsilon^{A\rightarrow B}(\Delta^{A\rightarrow B})$ which are continuous functionals of the trajectories of the scaled random walks and continuous limit processes respectively and such that $\J_{n,\varepsilon}^{A\rightarrow B}$ conveniently scaled converges in distribution to $\alpha_\varepsilon^{A\rightarrow B}(\Delta^{A\rightarrow B})$. Recalling that the space of probability measures over a Polish space endowed with the weak topology is itself Polish, then it becomes a consequence of the theory of Polish spaces that $\J_n^{A\rightarrow B}$ conveniently scaled converges in distribution to $\alpha^{A\rightarrow B}(\Delta^{A\rightarrow B})$ so long as we may also show that for any $n\ge1$,
\begin{equation}
\label{convergence:the link is close to the initial variable at n fixed}
    \J_{n,\varepsilon}^{A\rightarrow B}\underset{\varepsilon\rightarrow 0}\implies \J_n^{A\rightarrow B},
\end{equation}
and that 
\begin{equation}
    \label{convergence:the link is close to the initial variable uniformly strongly}
    \d_0\left( S(n)\J_{n,\varepsilon}^{A\rightarrow B},S(n)\J_n^{A\rightarrow B}\right)\underset{\substack{n\rightarrow\infty \\ \varepsilon\rightarrow 0}}\longrightarrow0,
\end{equation}
where $S$ is the natural scaling function for $\J_n^{A\rightarrow B}$, and $\d_0$ is a generic metric on the space of probability measures on $\RR$ such that $\d_0(\mu_n,\mu)\underset{n\rightarrow0}\longrightarrow0$ implies that $\mu_n$ converges weakly to $\mu$ (see \cite{DynkinSelf-IntersectionGauge1988}, Lemma $1.1$).

\begin{proposition}
\label{theorem:discrete ILT convergence in distribution to continuous ILT}
    Suppose that $|A\cup B|(d-\beta)<d$. Then
    \begin{equation*}
        \frac{b(n)^{d(|A\cup B|-1)}}{n^{|A\cup B|}}\mathcal J_n^{A\rightarrow B}\distribution \alpha^{A\cup B}\left(\Delta^{A\rightarrow B}\right).
    \end{equation*}
\end{proposition}

\begin{proof}
    In the proof, to simplify notations, we suppose that $A$ and $B$ are respectively of the form $[\![m]\!]$ and $[\![m+1,\widetilde m]\!]$ for some $m<\widetilde m\le N$. Letting $\TT^d\eqdef[-\pi,\pi]^d$ and noting that for all $x,y\in\RR^d$, $(2\pi)^{-d}\int_{\TT^d}\d\xi\exp(-\i\langle\xi,x-y\rangle)$ equals $1$ if $x=y$ and $0$ otherwise, we get
    \begin{equation*}
        \mathcal J_n^{A\rightarrow B}=\sum_{k\in\Delta_n^{A\rightarrow B}}\prod_{\ell=2}^{|A\cup B|}\frac{1}{(2\pi)^d}\int_{\TT^d}\d\xi_\ell\exp\left(-\i\left\langle\xi_\ell,X_{k_\ell}^{\ell}-X_{k_{\ell-1}}^{{\ell-1}}\right\rangle\right).
    \end{equation*}
    We now introduce an appropriately smoothed version of $\J_n^{A\rightarrow B}$. First, let $f$ be a radially symmetric function in the Schwartz space $\mathcal S(\RR^d)$ such that $\int_{\RR^d} f=1$, and for all $\varepsilon>0$, define 
    \begin{equation*}
        f_\varepsilon(x)\eqdef\varepsilon^{-d}f(\varepsilon^{-1}x),
    \end{equation*}
    in such a way that $f_\varepsilon$ converges to $\delta_0$ in the sense of distributions. For any $g\in\mathcal S(\RR^d)$ we let $\hat g\in\mathcal S(\RR^d)$ denote the Fourier transform of $g$. It is easily seen that $\hat{f_\varepsilon}(x)=\hat f(\varepsilon x)$. For any $\varepsilon>0$, we define
    \begin{equation*}
        \mathcal J_{n,\varepsilon}^{A\rightarrow B}\eqdef\sum_{k\in\Delta_n^{A\rightarrow B}}\prod_{\ell=2}^{|A\cup B|}\frac{1}{(2\pi)^d}\int_{\TT^d}\d\xi_\ell\exp\left(-\i\left\langle\xi_\ell,X_{k_\ell}^{\ell}-X_{k_{\ell-1}}^{{\ell-1}}\right\rangle\right)\hat f_\varepsilon \left(b(n)\xi_\ell\right).
    \end{equation*}
    By defining $\TT_n^d\eqdef b(n)\TT^d$, we see that
    \begin{align*}
        \mathcal J_{n,\varepsilon}^{A\rightarrow B}&=\sum_{k\in\Delta_n^{A\rightarrow B}}\prod_{\ell=2}^{|A\cup B|}\frac{1}{(2\pi)^d}\int_{\TT_n^d}\frac{\d\xi_\ell}{b(n)^d}\exp\left(-\i\left\langle\xi_\ell,\frac{X_{k_\ell}^{\ell}-X_{k_{\ell-1}}^{{\ell-1}}}{b(n)}\right\rangle\right)\hat f_\varepsilon(\xi_\ell)\\
        &=\sum_{k\in\Delta_n^{A\rightarrow B}}\prod_{\ell=2}^{|A\cup B|}\frac{1}{b(n)^d}\left(f_\varepsilon\left(\frac{X_{k_\ell}^{\ell}-X_{k_{\ell-1}}^{{\ell-1}}}{b(n)}\right)\right.\\
        &\quad\quad \quad\quad\quad\quad\quad\quad\quad\quad\quad\left.-\frac{1}{(2\pi)^d}\int_{(\TT_n^d)^c}\d\xi_\ell\exp\left(-\i\left\langle\xi_\ell,\frac{X_{k_\ell}^{\ell}-X_{k_{\ell-1}}^{{\ell-1}}}{b(n)}\right\rangle\right)\hat f_\varepsilon(\xi_\ell)\right),
    \end{align*}
    where we used Fourier inversion in the second equality. Since $f$ is radially symmetric, so is $\hat f$. Indeed, if we take $x,y\in\RR^d$ such that $|x|=|y|$ and let $\rho$ be a rotation that sends $x$ to $y$, then
    \begin{align*}
        \hat f(y)&=\hat f(\rho(x))=\frac{1}{(2\pi)^d}\int_{\RR^d}\d\xi\exp(\i\langle\xi,\rho(x)\rangle)f(\xi)\\
        &=\frac{1}{(2\pi)^d}\int_{\RR^d}\d\xi\exp(\i\langle\rho^{-1}(\xi),x\rangle)f(\xi)\\
        &=\frac{1}{(2\pi)^d}\int_{\RR^d}\d\xi\exp(\i\langle\xi,x\rangle)f(\rho(\xi))\\
        &=\frac{1}{(2\pi)^d}\int_{\RR^d}\d\xi\exp(\i\langle\xi,x\rangle)f(\xi)=\hat f(x).
    \end{align*}
    Therefore, by letting $\widetilde f$ denote the radial part of $\hat f$, we have $\widetilde f\in\mathcal S(\RR)$ whence for any $m\ge d+1$,
    \begin{align*}
        \left|\int_{(\TT_n^d)^c}\d\xi\exp\left(-\i\langle \xi,x\right\rangle)\hat f_\varepsilon(\xi)\right|&\le {d}\int_{b(n)\pi}^\infty\d r\ r^{d-1}\widetilde f(\varepsilon r)\\
        &\le C_{d,m}\varepsilon^{-m}\int_{b(n)\pi}^\infty\d r\ r^{d-m-1}\\
        &\le C_{d,m}\varepsilon^{-m}b(n)^{d-m}.
    \end{align*}
    Since the last estimate is uniform in $x$, we deduce that there is some $\kappa>0$ such that for any $m\ge d+1$,
    \begin{align*}
        \mathcal J_{n,\varepsilon}^{A\rightarrow B}&=\sum_{k\in\Delta_n^{A\rightarrow B}}\prod_{\ell=2}^{|A\cup B|}\frac{1}{b(n)^d}f_\varepsilon\left(\frac{X_{k_\ell}^{\ell}-X_{k_{\ell-1}}^{{\ell-1}}}{b(n)}\right)+\mathcal O\left(|\Delta_n^{A\rightarrow B}|b(n)^{\kappa(d-m)}\right)\\
        &=\frac{n^{|A\cup B|}}{b(n)^{d(|A\cup B|-1)}}\int_{\Delta^{A\rightarrow B}}\d t\prod_{\ell=2}^{|A\cup B|}f_\varepsilon\left(X_{t_\ell}^{\ell,n}-X_{t_{\ell-1}}^{\ell-1,n}\right)+\mathcal O\left(|\Delta_n^{A\rightarrow B}|b(n)^{\kappa(d-m)}\right).
    \end{align*}
    Noting that for a function $g$ defined on $\ZZ^{|A\cup B|}$, we have 
    \begin{align*}
        \sum_{k\in\Delta_n^{A\rightarrow B}}g(k)&=\sum_{\sigma\in|S|^{A\rightarrow B}}\sum_{k\in\Delta_n^\sigma}g(k)=\sum_{\sigma\in|S|^{A\rightarrow B}}\int_{n^{|A\cup B|}\Delta^\sigma}\d t\ g(\lfloor t\rfloor)=n^{|A\cup B|}\sum_{\sigma\in|S|^{A\rightarrow B}}\int_{\Delta^\sigma}\d t\ g(\lfloor nt\rfloor)\\
        &=n^{|A\cup B|}\int_{\Delta^{A\rightarrow B}}\d t\ g(\lfloor nt\rfloor),
    \end{align*}
    then
    \begin{equation*}
        \frac{b(n)^{d(|A\cup B|-1)}}{n^{|A\cup B|}}\mathcal J_{n,\varepsilon}^{A\rightarrow B}=\int_{\Delta^{A\rightarrow B}}\d t\prod_{\ell=2}^{|A\cup B|}f_\varepsilon\left(X_{t_\ell}^{\ell,n}-X_{t_{\ell-1}}^{\ell-1,n}\right)+\mathcal O_\varepsilon\left(\frac{|\Delta_n^{A\rightarrow B}|}{n^{|A\cup B|}}\frac{b(n)^{d(|A\cup B|-1)}}{b(n)^{\kappa(m-d)}}\right).
    \end{equation*}
    Since $(\omega^1,\dots,\omega^{|A\cup B|})\mapsto\int_{\Delta^{A\rightarrow B}}\d t\prod_{\ell=2}^{|A\cup B|}f_\varepsilon\left(\omega_{t_\ell}^\ell-\omega_{t_{\ell-1}}^{\ell-1}\right)$ is continuous for the topology of uniform convergence on compact sets of continuous paths, then (\ref{convergence:skorokhod's convergence theorem}) ensures that 
    \begin{equation*}
        \int_{\Delta^{A\rightarrow B}}\d t\prod_{\ell=2}^{|A\cup B|}f_\varepsilon\left(X_{t_\ell}^{\ell,n}-X_{t_{\ell-1}}^{\ell-1,n}\right)\distribution\int_{\Delta^{A\rightarrow B}}\d t\prod_{\ell=2}^{|A\cup B|}f_\varepsilon\left(U_{t_\ell}^\ell-U_{t_{\ell-1}}^{\ell-1}\right)=\alpha_\varepsilon^{A\cup B}(\Delta^{A\rightarrow B}).
    \end{equation*}
    Furthermore, it is clear that for $m$ sufficiently large, 
    \begin{equation*}
        \underset{n\rightarrow \infty}\lim\frac{|\Delta_n^{A\rightarrow B}|b(n)^{d(|A\cup B|-1)}}{n^{|A\cup B|}b(n)^{\kappa(m-d)}}=0,
    \end{equation*}
    whence 
    \begin{equation*}
        \frac{b(n)^{d(|A\cup B|-1)}}{n^{|A\cup B|}}\mathcal J_{n,\varepsilon}^{A\rightarrow B}\distribution\alpha_\varepsilon^{A\cup B}(\Delta^{A\rightarrow B}).
    \end{equation*}
    Furthermore, as shown in \cite{ChenRosen2005}, so long as $|A\cup B|(d-\beta)<d$, then for all $p\ge1$,
    \begin{equation}
    \label{convergence:a.s. and Lp convergence of smoothed ILT}
        \alpha_\varepsilon^{A\cup B}(\Delta^{A\rightarrow B})\underset{\varepsilon\rightarrow0}{\overset{a.s.,L^p}{\longrightarrow}}\alpha^{A\cup B}(\Delta^{A\rightarrow B}),
    \end{equation}
    thus it remains to show that $\mathcal J_n^{A\rightarrow B}$ and $\mathcal J_{n,\varepsilon}^{A\rightarrow B}$ conveniently scaled are close as $\varepsilon\rightarrow 0$ uniformly in $n$. To this end, borrowing the notations of Proposition \ref{proposition: h(n)^(a+b)R_n^(A->B) and J_n^(A->B) are close in L^2}, we show that there are $C,\delta>0$ such that for any $n\ge1$, $\varepsilon>0$, 
    \begin{equation}
    \label{ineq:J_n(A->B) and J_(n,varepsilon)(A->B) are close}
        \frac{b(n)^{2d\left(|A\cup B|-1\right)}}{n^{2|A\cup B|}}\EE\left[\left(\J_n^{A\rightarrow B}-\J_{n,\varepsilon}^{A\rightarrow B}\right)^2\right]\lesssim\varepsilon^\delta.
    \end{equation}
    This clearly implies both (\ref{convergence:the link is close to the initial variable at n fixed}) and (\ref{convergence:the link is close to the initial variable uniformly strongly}), since convergence in the $L^2$ norm implies convergence in distribution. Therefore, we have
    \begin{align}
    \label{eq:L^2 norm of the difference J_n^(A->B)-J_(n,eps)^(A->B)}
        &\mathcal E_{n,\varepsilon}\eqdef\frac{b(n)^{2d\left(|A\cup B|-1\right)}}{n^{2|A\cup B|}}\EE\left[\left(\J_n^{A\rightarrow B}-\J_{n,\varepsilon}^{A\rightarrow B}\right)^2\right]\nonumber\\
        &\hspace{2em}\lesssim\frac{1}{n^{2|A\cup B|}}\sum_{k\in\widetilde\Delta_n^{A\rightarrow B}}\int\d\xi^1\d\xi^2\EE\left[\prod_{\ell=2}^{|A\cup B|}\exp\left(-\i\sum_{j=1}^2\left\langle\xi_\ell^j,b(n)^{-1}\left(X_{k_\ell^j}^\ell-X_{k_{\ell-1}^j}^{\ell-1}\right)\right\rangle\right)\right]\nonumber\\
        &\hspace{20em}\times\left(1-\prod_{\ell=2}^{|A\cup B|}\hat f_\varepsilon\left(\xi_\ell^1\right)\right)\left(1-\prod_{\ell=2}^{|A\cup B|}\hat f_\varepsilon\left(\xi_\ell^2\right)\right),
    \end{align}
    where the integral is taken over $(\xi_\ell^j)_{2\le\ell\le|A\cup B|}^{j=1,2}\in\TT_n^{2d(|A\cup B|-1)}\times\TT_n^{2d(|A\cup B|-1)}=:\mathcal T_n$. By letting $\xi_1^j=\xi_{|A\cup B|+1}^j=0$ for $j=1,2$, we notice that
    \begin{align*}
        \prod_{\ell=2}^{|A\cup B|}\exp\left(-\i\sum_{j=1}^2\left\langle\xi_\ell^j,b(n)^{-1}\left(X_{k_\ell^j}^\ell-X_{k_{\ell-1}^j}^{\ell-1}\right)\right\rangle\right)&=\exp\left(-\i\sum_{j=1}^2\sum_{\ell=2}^{|A\cup B|}\left\langle\xi_\ell^j,b(n)^{-1}\left(X_{k_\ell^j}^\ell-X_{k_{\ell-1}^j}^{\ell-1}\right)\right\rangle\right)\\
        &=\exp\left(-\i\sum_{j=1}^2\sum_{\ell=1}^{|A\cup B|}\left\langle\xi_\ell^j-\xi_{\ell+1}^j,b(n)^{-1}X_{k_\ell^j}^\ell\right\rangle\right),
    \end{align*}
    whence the expectation in (\ref{eq:L^2 norm of the difference J_n^(A->B)-J_(n,eps)^(A->B)}) is equal to 
    \begin{equation}
    \label{eq:expectation term in L^2 norm of error}
        \prod_{\ell=1}^{|A\cup B|}\EE\left[\exp\left(-\i\sum_{j=1}^2\left\langle\xi_\ell^j-\xi_{\ell+1}^j,b(n)^{-1}X_{k_\ell^j}^\ell\right\rangle\right)\right].
    \end{equation}
    Once again, we decompose $\sum_{\widetilde\Delta_n^{A\rightarrow B}}$ into $\sum_{S\subseteq A\cup B}\sum_{\widetilde\Delta_{n,S}^{A\rightarrow B}}$ and by symmetry focus only on the case $S=A\cup B$, $i.e.$ $k_\ell^1\le k_\ell^2$ for every $\ell\in A\cup B$, the others being equivalent. By letting $k\in\widetilde\Delta_{n,A\cup B}^{A\rightarrow B}$ and $\ell\in A\cup B$, we get
    \begin{align}
       & \EE\left[\exp\left(-\i\sum_{j=1}^2\left\langle\xi_\ell^j-\xi_{\ell+1}^j,b(n)^{-1}X_{k_\ell^j}^\ell\right\rangle\right)\right]\nonumber\\
        &\hspace{3em}=\EE\left[\exp\left(-\i\left\langle\xi_\ell^{2}-\xi_{\ell+1}^{2},b(n)^{-1}X_{k_\ell^{2}}^\ell\right\rangle-\i\left\langle\xi_\ell^{1}-\xi_{\ell+1}^{1},b(n)^{-1}X_{k_\ell^{1}}^\ell\right\rangle\right)\right]\nonumber\\
        &\hspace{3em}=\EE\left[\exp\left(-\i\left\langle\xi_\ell^{2}-\xi_{\ell+1}^{2},b(n)^{-1}\left(X_{k_\ell^{2}}^\ell-X_{k_\ell^{1}}^\ell\right)\right\rangle\right)\right]\nonumber\\
        &\hspace{6em}\times\EE\left[\exp\left(-\i\left\langle\sum_{j=1}^2(\xi_\ell^j-\xi_{\ell+1}^j),b(n)^{-1}X_{k_\ell^{1}}^\ell\right\rangle\right)\right]\nonumber\\
        &\hspace{3em}=\varphi\left(\frac{\xi_\ell^{2}-\xi_{\ell+1}^{2}}{b(n)}\right)^{k_\ell^{2}-k_\ell^{1}}\varphi\left(\sum_{j=1}^2\frac{\xi_\ell^j-\xi_{\ell+1}^j}{b(n)}\right)^{k_\ell^{1}}\label{eq:the expectation term in error is a product of characteristic functions},
    \end{align}
    where $\varphi$ is the characteristic function of the $X^\ell$. By \eqref{ineq:riemann sum of scaled characteristic is bounded by a power of x}, for any $(\xi^1,\xi^2)\in\mathcal T_n$, recalling that $\overline x\equiv x\pmod{\pi b(n)}$ we have 
    \begin{align}
        &\frac{1}{n^{2|A\cup B|}}\sum_{k\in\widetilde\Delta_{n,\rho}^{A\rightarrow B}}\prod_{\ell=1}^{|A\cup B|}\left|\varphi\left(\frac{\xi_\ell^{2}-\xi_{\ell+1}^{2}}{b(n)}\right)^{k_\ell^{2}-k_\ell^{1}}\varphi\left(\sum_{j=1}^2\frac{\xi_\ell^j-\xi_{\ell+1}^j}{b(n)}\right)^{k_\ell^{1}}\right|\nonumber\\
        &\hspace{3em}\le \prod_{\ell=1}^{|A\cup B|}\left(\frac{1}{n}\sum_{k=0}^n\left|\varphi\left(\frac{\xi_\ell^{2}-\xi_{\ell+1}^{2}}{b(n)}\right)\right|^k\right)\left(\frac{1}{n}\sum_{k=0}^n\left|\varphi\left(\sum_{j=1}^2\frac{\xi_\ell^j-\xi_{\ell+1}^j}{b(n)}\right)\right|^k\right)\nonumber\\
        &\hspace{3em}\lesssim\prod_{\ell=1}^{|A\cup B|}\left(1+\left|\overline{\xi_\ell^{2}}-\overline{\xi_{\ell+1}^{2}}\right|^{\beta(1-\varepsilon)}\right)^{-1}\left(1+\left|\overline{\xi_\ell^{1}}-\overline{\xi_{\ell+1}^{1}}+\overline{\xi_\ell^{2}}-\overline{\xi_{\ell+1}^{2}}\right|^{\beta(1-\varepsilon)}\right)^{-1}.\label{ineq:bound for the expectation in error calculation}
    \end{align}
    In the rest of the proof, for all $\ell\in A\cup B$, we shall write
    \begin{equation*}
        \zeta_\ell^{1}\eqdef\xi_\ell^{2}-\xi_{\ell+1}^{2},\quad\zeta_\ell^{2}\eqdef\xi_\ell^{1}-\xi_{\ell+1}^{1}+\xi_\ell^{2}-\xi_{\ell+1}^{2}.
    \end{equation*}
    Note that if $(\xi^1,\xi^2)\in\mathcal T_n$, then $(\zeta^{1},\zeta^{2})\in4\mathcal T_n$. By letting 
    \begin{equation*}
        F(\varepsilon,\xi^1,\xi^2)\eqdef\left(1-\prod_{\ell=2}^{|A\cup B|}\hat f_\varepsilon\left(\xi_\ell^1\right)\right)\left(1-\prod_{\ell=2}^{|A\cup B|}\hat f_\varepsilon\left(\xi_\ell^2\right)\right),
    \end{equation*}
    then by using successively the identity
    \begin{equation*}
        \prod_{k=1}^qa_k=\prod_{j=1}^q\left(\prod_{k\ne j}a_k^{\frac{1}{q-1}}\right)
    \end{equation*}
    and Hölder's inequality, we get for all $\varepsilon>0$,
    \begin{equation*}
        \mathcal E_{n,\varepsilon}\lesssim\prod_{\ell=1}^{|A\cup B|}\left(\int_{\mathcal T_n}\d\xi^1\d\xi^2 F(\varepsilon,\xi^1,\xi^2)\prod_{m\ne \ell}\frac{1}{\left(1+|\overline{\zeta_m^1}|^{\beta(1-\varepsilon)}\right)^{\frac{|A\cup B|}{|A\cup B|-1}}\left(1+|\overline{\zeta_m^2}|^{\beta(1-\varepsilon)}\right)^\frac{|A\cup B|}{|A\cup B|-1}}\right)^{\frac{1}{|A\cup B|}}
    \end{equation*}
    It remains to obtain a bound on $F(\varepsilon,\xi^1,\xi^2)$. Recalling that $\hat f_\varepsilon(\cdot)\equiv\hat f(\varepsilon\cdot)$, then by Lemma \ref{ineq:bound for products of Fourier transform},
    \begin{equation*}
        \left|1-\prod_{\ell=2}^{|A\cup B|}\hat f_{\varepsilon}\left(\xi_\ell\right)\right|\lesssim \varepsilon^\delta\sum_{\ell=2}^{|A\cup B|}|\xi_\ell|^\delta,
    \end{equation*}
    whence for all $(\xi^1,\xi^2)\in\mathcal T_n$ and $\ell\in[\![|A\cup B|]\!]$,
    \begin{align}
    \label{ineq:bound for the term in epsilon}
        F(\varepsilon,\xi^1,\xi^2)&\lesssim\varepsilon^{2\delta}\sum_{\ell_1,\ell_2=2}^{|A\cup B|}|\xi_{\ell_1}^1|^\delta|\xi_{\ell_2}^2|^\delta\nonumber\\
        &\lesssim\varepsilon^{2\delta}\sum_{\ell_1,\ell_2\ne \ell}\left(|\zeta_{\ell_1}^1|^{2\delta}+|\zeta_{\ell_2}^2|^{2\delta}\right),
    \end{align}
    since $\xi_j^1,\xi_j^2\in\mathrm{span}\left\{\zeta_k^1,\zeta_k^2, k\ne \ell\right\}$ for any $\ell\in A\cup B$. Injecting (\ref{ineq:bound for the expectation in error calculation}) and (\ref{ineq:bound for the term in epsilon}) into (\ref{eq:L^2 norm of the difference J_n^(A->B)-J_(n,eps)^(A->B)}) yields
    \begin{equation}
    \label{ineq:final bound for E_(n,varepsilon)}
        \mathcal E_{n,\varepsilon}\lesssim \varepsilon^{2\delta}\prod_{\ell=1}^{|A\cup B|}\left(\int_{4\mathcal T_n}\d\zeta \sum_{\ell_1,\ell_2\ne \ell}\prod_{m\ne \ell}\frac{\left(|\overline{\zeta_{\ell_1}^1}|^{2\delta}+|\overline{\zeta_{\ell_2}^2}|^{2\delta}\right)^{\frac{1}{|A\cup B|-1}}}{\left(1+|\overline{\zeta_m^1}|^{\beta(1-\varepsilon)}\right)^{\frac{|A\cup B|}{|A\cup B|-1}}\left(1+|\overline{\zeta_m^2}|^{\beta(1-\varepsilon)}\right)^\frac{|A\cup B|}{|A\cup B|-1}}\right)^{\frac{1}{|A\cup B|}}.
    \end{equation}
    
    In particular, looking at the integral in (\ref{ineq:final bound for E_(n,varepsilon)}) component by component, we may deduce that it is uniformly bounded in $n\ge1$ once we show that 
    \begin{equation}
        \label{ineq:uniform in n bound on E_(n,varepsilon)}
        \sup_{n\ge1}\int_{4\TT_n^{d}}\d\zeta\left(\frac{|\overline\zeta|^{2\delta}}{\left(1+|\overline\zeta|^{\beta(1-\varepsilon)}\right)^{|A\cup B|}}\right)^{\frac{1}{|A\cup B|-1}}<\infty.
    \end{equation}
    This is indeed the case so long as we take $\delta,\varepsilon>0$ such that 
    \begin{equation*}
        \frac{\beta(1-\varepsilon)|A\cup B|-2\delta}{|A\cup B|-1}>d\iff |A\cup B|(d-\beta)<d-\varepsilon\beta|A\cup B|-2\delta,
    \end{equation*}
    which is attainable by letting $\varepsilon,\delta$ become sufficiently small since we supposed $|A\cup B|(d-\beta)<d$. For such a choice of $\varepsilon,\delta$, we thus have
    \begin{equation*}
        \mathcal E_{n,\varepsilon}\lesssim\varepsilon^{2\delta},
    \end{equation*}
    which is sufficient to conclude.
\end{proof}
\begin{remark}
        The method of proof used here actually yields the seemingly stronger result
    \begin{equation*}
        \left(\left(X^{s,n}\right)_{s\in A\cup B},\frac{b(n)^{d(|A\cup B|-1)}}{n^{|A\cup B|}}\mathcal J_n^{A\rightarrow B}\right)\distribution\left(\left(U^s\right)_{s\in A\cup B},\alpha^{A\cup B}\left(\Delta^{A\rightarrow B}\right)\right),
    \end{equation*}
    where the convergence of processes occurs in the $J_1$ topology. Indeed, since $\J_{n,\varepsilon}^{A\rightarrow B}$ is, up to some arbitrarily small error term, a continuous functional of the $\left(X^{s,n}\right)_{s\in A\cup B}$, then 
    \begin{equation*}
        \left(\left(X^{s,n}\right)_{s\in A\cup B},\frac{b(n)^{d(|A\cup B|-1)}}{n^{|A\cup B|}}\mathcal J_{n,\varepsilon}^{A\rightarrow B}\right)\distribution\left(\left(U^s\right)_{s\in A\cup B},\alpha_\varepsilon^{A\cup B}\left(\Delta^{A\rightarrow B}\right)\right).
    \end{equation*}
    This combined with (\ref{convergence:a.s. and Lp convergence of smoothed ILT}) and (\ref{ineq:J_n(A->B) and J_(n,varepsilon)(A->B) are close}) yields the result.
    \end{remark}

\subsection{Joint convergence and convergence of moments}

Firstly, note that Theorem \ref{theorem:convergence of R_n^(A->B)} follows at once from Propositions \ref{proposition: h(n)^(a+b)R_n^(A->B) and J_n^(A->B) are close in L^2} and \ref{theorem:discrete ILT convergence in distribution to continuous ILT}. Thanks to the remark following the proof of Proposition \ref{theorem:discrete ILT convergence in distribution to continuous ILT}, then we have actually shown the joint convergence
\begin{equation*}
    \left((X^{s,n})_{s\in{A\cup B}}, \frac{h(n)^{|A\cup B|}b(n)^{d(|A\cup B|-1)}}{n^{|A\cup B|}}\mathcal R_n^{A\rightarrow B}\right)\distribution\left((U^s)_{s\in A\cup B}, \alpha^{A\cup B}\left(\Delta^{A\rightarrow B}\right)\right).
\end{equation*}
Furthermore, note that the method of proof of the CLT in \cite{LeGallRosenRangeOfStableRandomWalks1991} also yields the joint convergence
\begin{equation*}
    \left( X^{k,n},\frac{h(n)^2b(n)^d}{n^2}\left(\left|X^k(0,n)\right|-\EE\left[\left|X^k(0,n)\right|\right]\right)\right)\distribution \left(U^k,-\gamma^k([0,1]_\le^2)\right).
\end{equation*}
Furthermore, the proof of Proposition \ref{proposition: h(n)^(a+b)R_n^(A->B) and J_n^(A->B) are close in L^2} shows that the sequence $\frac{h(n)^{|A\cup B|}b(n)^{d(|A\cup B|-1)}}{n^{|A\cup B|}}\mathcal R_n^{A\rightarrow B}$ is bounded in $L^2$, and so it is uniformly integrable. In particular, this together with Theorem \ref{theorem:convergence of R_n^(A->B)} yields the convergence of means
\begin{equation*}
    \frac{h(n)^{|A\cup B|}b(n)^{d(|A\cup B|-1)}}{n^{|A\cup B|}}\EE\left[\mathcal R_n^{A\rightarrow B}\right]\longrightarrow\EE\left[\alpha^{A\cup B}\left(\Delta^{A\rightarrow B}\right)\right].
\end{equation*}
Combining these previous remarks, we arrive at the following joint convergence result
\begin{corollary}
    \label{corollarry:joint convergence of range and competitive range}
    The vector
    \begin{equation*}
        \frac{h(n)^2b(n)^d}{n^2}\left(\left|X^k(0,n)\right|-\EE\left[\left|X^k(0,n)\right|\right],\R_n^{j\rightarrow k}-\EE\left[\R_n^{j\rightarrow k}\right]\right)_{1\le j\ne k\le N}
    \end{equation*}
    converges in distribution to the vector
    \begin{equation*}
        \left(-\gamma^k([0,1]_\le^2), \alpha^{j,k}([0,1]_\le^2)-\EE\left[\alpha^{j,k}([0,1]_\le^2)\right]\right).
    \end{equation*}
\end{corollary}

\section{Proof of Theorems \ref{theorem:LLN} and \ref{theorem:CLT}}
\label{section:proof of main results}

We are now ready to give a proof of the announced LLN and CLT under \ref{Assumption A1}, \ref{Assumption A2}, \ref{Assumption A3}. We start by proving a decomposition formula for $\R_n^k$.

\begin{proposition}
\label{proposition:decomposition formula for mathcal R_n^k}
    For any $n\ge1$, $k\in[\![N]\!]$,
    \begin{equation}
    \label{eq:decomposition formula for R_n^k}
        \R_n^k=\left|X^k(0,n)\right|-\sum_{j\ne k}\R_n^{j\rightarrow k}+\varepsilon_n^k,
    \end{equation}
    where $\varepsilon_n^k$ is an error term such that 
    \begin{equation}
        \label{limit:E[(epsilon_n^k)^2] is small}
        \lim_{n\rightarrow\infty}\frac{h(n)^4b(n)^{2d}}{n^4}\EE\left[\left(\varepsilon_{n}^k\right)^2\right]=0.
    \end{equation}
\end{proposition}

\begin{proof}
    This formula is a simple consequence of the inclusion-exclusion principle. Indeed,
    \begin{align*}
        \R_n^k&=\sum_{x\in\ZZ^d}\1\left(T_x^k\le n,T_x^k<\underset{j\ne k}\min\  T_x^j\right)\\
        &=\sum_{x\in\ZZ^d}\1\left(T_x^k\le n\right)\prod_{j\ne k}\left(1-\1\left(T_x^j\le T_x^k\right)\right)\\
        &=\sum_{x\in\ZZ^d}\1\left(T_x^k\le n\right)-\sum_{j\ne k}\1\left(T_x^k\le n\right)\1\left(T_x^j\le T_x^k\right)\\
        &\hspace{10em}+\sum_{\substack{S\subseteq[\![N]\!]\backslash\{k\} \\ |S|\ge2}}(-1)^{|S|}\sum_{x\in\ZZ^d}\1\left(T_x^k\le n\right)\prod_{s\in S}\1\left(T_x^s\le T_x^k\right)\\
        &=\left|X^k(0,n)\right|-\sum_{j\ne k}\R_n^{j\rightarrow k}+\sum_{\substack{S\subseteq[\![N]\!]\backslash\{k\} \\ |S|\ge2}}(-1)^{|S|}\R_n^{S\rightarrow k}.
    \end{align*}
    We conclude by letting $\varepsilon_n^k$ denote the last sum in the previous expression. To see why (\ref{limit:E[(epsilon_n^k)^2] is small}) holds, it suffices to notice that for any $S\subseteq[\![N]\!]\backslash\{k\}$ with $|S|\ge2$ and for any $a,b\in S$, $a\ne b$, one has $\R_n^{S\rightarrow k}\le I_n^{\{a,b,k\}}$ and then apply Proposition \ref{proposition:bounds on moments of intersections} and Lemma \ref{lemma:limit of h(n)b(n)^d/n}.
\end{proof}
\begin{proof}[Proof of Theorem \ref{theorem:LLN}]
    By (\ref{eq:decomposition formula for R_n^k}), we have
\begin{equation*}
    \frac{h(n)}n\R_n^k=\frac{h(n)}n \left|X^k(0,n)\right|+\widetilde\varepsilon_n^k,
\end{equation*}
where 
\begin{equation*}
    \widetilde\varepsilon_n^k\eqdef\frac{h(n)}n\sum_{\substack{S\subseteq[\![N]\!]\backslash\{k\} \\ |S|\ge 1}}(-1)^{|S|}\R_n^{S\rightarrow k}.
\end{equation*}
By Theorem $6.9$ of \cite{LeGallRosenRangeOfStableRandomWalks1991}, the first term converges in $L^2$ (and $a.s.$ if $s(n)\ge1$ for any $n\ge1$) to $1$ and by (\ref{convergence:limit of R_n^(A->B)}), we have for all $S\subseteq[\![N]\!]\backslash\{k\}$ with $|S|\ge1$,
\begin{equation*}
    \frac{h(n)}{n}\R_n^{S\rightarrow k}=\left(\frac{n}{h(n)b(n)^d}\right)^{|S|}\left(\frac{h(n)^{|S|+1}b(n)^{d|S|}}{n^{|S|+1}}\R_n^{S\rightarrow k}\right)\eqdef a_nY_n.
\end{equation*}
According to Theorem \ref{theorem:convergence of R_n^(A->B)} and Lemma \ref{lemma:limit of h(n)b(n)^d/n}, we have respectively that $Y_n$ has converging second moments and $a_n$ converges to $0$ as $n\rightarrow\infty$. In particular, $a_nY_n$ converges to $0$ in $L^2$, whence $\widetilde\varepsilon_n^k$ too. It remains to obtain $a.s.$ convergence to $0$. To this end, we let $\varepsilon>0$, $p\ge1$ and write
\begin{align*}
    \sum_{n\ge1}\PP\left(|a_n Y_n|\ge\varepsilon\right)\le \varepsilon^{-p}\sum_{n\ge1}a_n^p\EE[Y_n^p]\lesssim\sum_{n\ge1}a_n^p,
\end{align*}
where in the last inequality we used Proposition \ref{proposition:bounds on moments of intersections} and the fact that $\R_n^{S\rightarrow k}\le I_n^{S\cup\{k\}}$ $a.s.$ When $d>\beta$, by writing 
\begin{equation*}
    a_n=\frac{1}{h(n)^{|S|}s(n)^{d|S|}}\frac{1}{n^{|S|(d/\beta-1)}},
\end{equation*}
then by the Potter bounds it is clear that we may choose $p$ large enough so that $\sum_{n\ge1}a_n^p<\infty$, and so by the Borel-Cantelli lemma, $a_nY_n\overset{a.s.}\longrightarrow0$ as $n\rightarrow\infty$. It remains to treat the case $d=\beta$. In the proof of Theorem $6.9$ in \cite{LeGallRosenRangeOfStableRandomWalks1991}, it is shown that for any $r>1$, 
\begin{equation*}
    \sum_{k\ge1}\frac{1}{h(\lfloor r^k\rfloor)^2s(\lfloor r^k\rfloor)^{2d}}<\infty,
\end{equation*}
and so by the same reasoning as before, this shows that $a_{\lfloor r^k\rfloor}Y_{\lfloor r^k\rfloor}$ converges $a.s.$ to $0$ as $k\rightarrow\infty$. Fixing $r>1$ and for every $n\ge1$, letting $k(n)$ denote the unique integer such that $\lfloor r^{k(n)}\rfloor\le n<\lfloor r^{k(n)+1}\rfloor$, then we have on one hand 
\begin{equation*}
    \R_{\lfloor r^{k(n)}\rfloor}^{S\rightarrow k}\le\R_n^{S\rightarrow k}\le\R_{\lfloor r^{k(n)+1}\rfloor},
\end{equation*}
and on the other for any $\varepsilon>0$,
\begin{equation*}
    \frac{h(n)\lfloor r^{k(n)}\rfloor}{nh(\lfloor r^{k(n)}\rfloor)}\le \frac{h(\lfloor r^{k(n)+1}\rfloor)}{h(\lfloor r^{k(n)}\rfloor)}\lesssim r^{\varepsilon},
\end{equation*}
and
\begin{equation*}
    \frac{h(n)\lfloor r^{k(n)}\rfloor}{nh(\lfloor r^{k(n)}\rfloor)}\ge\frac{\lfloor r^{k(n)}\rfloor}{\lfloor r^{k(n)+1}\rfloor}\ge r^{-1}.
\end{equation*}
In particular, 
\begin{equation*}
    \lim_{n\rightarrow\infty}\frac{h(n)}{n}\R_{\lfloor r^{k(n)}\rfloor}^{S\rightarrow k}=\lim_{n\rightarrow\infty}\frac{h(n)}{n}\R_{\lfloor r^{k(n)+1}\rfloor}^{S\rightarrow k}=0,
\end{equation*}
whence $\frac{h(n)}{n}\R_n^{S\rightarrow k}$ converges $a.s.$ to $0$ as $n\rightarrow0$, which is sufficient to conclude.
\end{proof}

We now turn to the proof of Theorem \ref{theorem:CLT}.  

\begin{proof}[Proof of Theorem \ref{theorem:CLT}]
To simplify expressions, we introduce the notation $R_n^k:=\left|X^k(0,n)\right|$. Thanks to Proposition \ref{proposition:decomposition formula for mathcal R_n^k}, we have 
\begin{equation*}
    \frac{h(n)^2b(n)^d}{n^2}\overline\R_n^k=\frac{h(n)^2b(n)^d}{n^2}\left(\overline R_n^k-\sum_{j\ne k}\overline \R_n^{j\rightarrow k}\right)+o_{L^2}(1).
\end{equation*}
Thanks to Corollary \ref{corollarry:joint convergence of range and competitive range}, we have 
\begin{equation*}
    \frac{h(n)^2b(n)^d}{n^2}\left(\overline R_n^k-\sum_{j\ne k}\overline \R_n^{j\rightarrow k}\right)\distribution -\left(\gamma^k+\sum_{j\ne k}\overline\alpha^{j,k}\right)([0,1]_\le^2),
\end{equation*}
and we conclude by applying Slutsky's Theorem.
\end{proof}

\section{Discussion about the remaining cases}

Throughout the paper, we focused exclusively on the case $d/\beta\in[1,3/2)$, because it was the most interesting from the point of view of behavior of the fluctuations and also for our motivations. It is natural however to wonder what happens when $d/\beta\notin[1,3/2)$.

For instance, when $d/\beta\ge3/2$, the CLT established in \cite{LeGallRosenRangeOfStableRandomWalks1991} states that one may find a constant $\sigma>0$ depending on $d/\beta$ and the distribution of $X$ such that 
\begin{equation}
\label{limit:fluctuations range d/beta>=3/2}
    \frac{1}{\sqrt{ng(n)}}(|X(0,n)|-\EE[|X(0,n)|])\distribution\mathcal N(0,\sigma^2),
\end{equation}
where $g$ is the function defined as $g(n)\eqdef\sum_{k=1}^nk^2b(k)^{-2d}=\sum_{k=1}^nk^{2-2d/\beta}s(k)^{-2d}$. In particular, note that when $d/\beta>3/2$, $g(n)$ converges to some positive constant as $n\rightarrow\infty$ in which case we recover usual diffusive fluctuations, and when $d/\beta=3/2$ it can still be the case that $g(n)$ converges to a constant, but it is also possible that $g(n)$ diverges slowly in the sense of regular variation to $+\infty$. The competitive range satisfies the same CLT.
\begin{proposition}
    If $X^1,\dots,X^N$ are in the domain of attraction of a $\beta$-stable process $U$ with scale function $b(n)$ and that $d/\beta\ge3/2$, then there is an explicit constant $\sigma>0$ such that
    \begin{equation*}
    \frac{1}{\sqrt{ng(n)}}\left(\R_n^k-\EE\left[\R_n^k\right]\right)\distribution\mathcal N(0,\sigma^2).
    \end{equation*}
Furthermore, we have the joint convergence
\begin{equation*}
    \frac{1}{\sqrt{ng(n)}}\left(\R_n^k-\EE\left[\R_n^k\right]\right)_{k\in[\![N]\!]}\distribution\mathcal N(0,\sigma^2)\otimes\dots\otimes\mathcal N(0,\sigma^2).
    \end{equation*}
\end{proposition}

\begin{proof}
    To see that this is indeed the case, we use (\ref{eq:decomposition formula for R_n^k}) to write
    \begin{equation*}
        \R_n^k=\left|X^k(0,n)\right|+\widetilde\varepsilon_n^k,
    \end{equation*}
    where $\widetilde\varepsilon_n^k\eqdef\sum_{\substack{S\subseteq[\![N]\!]\backslash\{k\}\\ S\ne\emptyset}}(-1)^{|S|}\R_n^{S\rightarrow k}$. According to \cite{LeGallRosenRangeOfStableRandomWalks1991}, Theorem 4.6, there is some $\sigma>0$ such that 
    \begin{equation*}
        \Var\left(\left|X^k(0,n)\right|\right)\underset{n\rightarrow\infty}\sim \sigma^2ng(n).
    \end{equation*}
    Furthermore, for any non-empty $S\subseteq[\![N]\!]$ we let $a_S\in S$ and so
    \begin{align*}
        \Var\left(\widetilde\varepsilon_n^k\right)&\le\EE\left[(\widetilde\varepsilon_n^k)^2\right]\\
        &=\sum_{\substack{S_1,S_2\subseteq[\![N]\!]\backslash\{k\}\\ S_1,S_2\ne\emptyset}}(-1)^{|S_1|+|S_2|}\EE\left[\R_n^{S_1\rightarrow k}\R_n^{S_2\rightarrow k}\right]\\
        &\lesssim\EE\left[\left(I_n^{\{k,k'\}}\right)^2\right]\lesssim\EE\left[I_n^{\{k,k'\}}\right]^2,
    \end{align*}    
    where $k'\ne k$ and the constant depends only on $N$. According to \cite{LeGallRosenRangeOfStableRandomWalks1991}, Corollary 3.2, we have $\EE\left[I_n^{\{k,k'\}}\right]\lesssim n^2b(n)^{-d}$, and so
    \begin{equation*}
        \frac{\Var\left(\widetilde\varepsilon_n^k\right)}{ng(n)}\lesssim \frac{n^3}{b(n)^{2d}g(n)}=\frac{n^{2(3/2-d/\beta)}}{s(n)^{2d}g(n)}=o(1),
    \end{equation*}
    where the $o(1)$ is obvious in the case $d/\beta>3/2$ and is a consequence of properties of regularly varying functions in the case $d/\beta=3/2$. In particular, 
    \begin{equation*}
        \frac{\widetilde\varepsilon_n^k-\EE\left[\widetilde\varepsilon_n^k\right]}{\sqrt{ng(n)}}\underset{n\rightarrow\infty}{\overset{L^2}\longrightarrow}0,
    \end{equation*}
    and so we obtain the desired CLT by using \eqref{limit:fluctuations range d/beta>=3/2} and Slutsky's theorem. The fact that the joint convergence also holds follows from the same argument and the fact that the $X^k$ are independent.
\end{proof}

When $d/\beta<1$, the situation is different as the range no longer satisfies a strong LLN. Indeed, it was shown in \cite{LeGallRosenRangeOfStableRandomWalks1991} that
\begin{equation*}
    \frac{1}{b(n)}|X(0,n)|\distribution\lambda\left(U(0,1)\right),
\end{equation*}
where $\lambda$ denotes the Lebesgue measure on $\RR$. We believe that the competitive range behaves similarly in this regime.
\begin{conjecture}
    If $X^1,\dots,X^N$ are in the domain of attraction of a $\beta$-stable process $U$ with scale function $b(n)$ and that $d/\beta<1$, then 
    \begin{equation*}
        \frac{1}{b(n)}\R_n^k\distribution\lambda\left(\left\{x\in\RR, \tau_x^k\le 1\wedge\min_{j\ne k}\tau_x^j\right\}\right),
    \end{equation*}
    where $U^1,\dots,U^N$ are $i.i.d.$ copies of $U$ and where $\tau_x^\ell\eqdef\inf\{t>0, U_t^\ell=x\}$.

    Furthermore, joint convergence in $k\in[\![N]\!]$ also holds.
\end{conjecture}
We give a proof of this conjecture in the case where $X^1,\dots,X^N$ are symmetric simple random walks, however it is clear to see that the technique used can't be generalized to the heavy-tailed setting, since continuity of the scaling limit is crucial. We first prove the following lemma 

\begin{lemma}
\label{lemma:joint continuity of Brownian hitting times}
    Let $\mathbb W$ denote the Wiener measure on $\mathcal C\eqdef\mathcal C([0,1])$. We define a mapping $\tau$ on $\RR\times\mathcal C$ by 
    \begin{equation*}
        \tau(x,\omega)\eqdef\inf\{t\ge0, \omega_t=x\},
    \end{equation*}
    with the convention $\inf\emptyset=\infty$. Then for  $\lambda\otimes\mathbb W$-a.a. $(x,\omega)\in\RR\times\mathcal C$, $\tau$ is continuous at $(x,\omega)$ for the product uniform topology.
\end{lemma}

\begin{proof}
    We start by noticing that $\tau$ is discontinuous on the set
    \begin{equation*}
        \mathcal A\eqdef\left\{(x,\omega)\in\RR\times\mathcal C, x\in\{\inf\omega,\sup\omega\}\right\}.
    \end{equation*}
    Indeed, for any $\omega\in\mathcal C$, $\tau(x,\omega)\le1$ for $x\in[\inf\omega,\sup\omega]$ but $\tau(x,\omega)=\infty$ for $x\notin[\inf\omega,\sup\omega]$. However, $\lambda\otimes\mathbb W(\mathcal A)=0$, and so we restrict our attention to $\mathcal A^c=\mathcal A_1\cup\mathcal A_2$, where 
    \begin{equation*}
        \mathcal A_1\eqdef\left\{(x,\omega)\in\mathcal A^c, x\in[\inf\omega,\sup\omega]\right\},\quad \mathcal A_2\eqdef\left\{(x,\omega)\in\mathcal A^c, x\notin[\inf\omega,\sup\omega]\right\}.
    \end{equation*}
    Since $\tau$ is identically equal to $\infty$ on $\mathcal A_2$, it is in particular continuous (by viewing $\infty$ as a cemetery point). Therefore, it remains to show that $\tau$ is continuous on $\mathcal A_1$.
    
    Let $(x,\omega)\in\mathcal A_1$ such that $\tau\eqdef\tau(x,\omega)<\infty$. By symmetry, we suppose that $\omega(0)<x$, and that for any $\varepsilon>0$, $\exists t_\varepsilon\in(\tau,\tau+\varepsilon)$ such that $\omega(t_\varepsilon)>x$. We can make this assumption without loss of generality, since for $\lambda\otimes\mathbb W$-a.a. $(x,\omega)\in\RR\times\mathcal C$, $\omega$ doesn't reach a local maxima at $\tau$. Let $\varepsilon\in(0,\tau)$. Since $\omega$ is continuous, we have 
    \begin{equation*}
        m_\varepsilon\eqdef\max_{t\le\tau-\varepsilon}\omega(t)<x,
    \end{equation*}
    and $M_\varepsilon\eqdef\omega(t_\varepsilon)>x$. Finally, let $\delta_\varepsilon\eqdef10^{-1}\left((x-m_\varepsilon)\wedge (M_\varepsilon-x)\right)$ and consider $(\widetilde x,\widetilde\omega)\in\RR\times\mathcal C$ such that $|x-\widetilde x|<\delta_\varepsilon$ and $\lVert\omega-\widetilde\omega\lVert_\infty<\delta_\varepsilon$. Then for any $t\le\tau-\varepsilon$,
    \begin{equation*}
        \widetilde\omega(t)\le\omega(t)+\delta_\varepsilon\le m_\varepsilon+\delta_\varepsilon<x-\delta_\varepsilon<\widetilde x.
    \end{equation*}
    In particular, $\tau(\widetilde x,\widetilde\omega)>\tau-\varepsilon$. On the other hand,
    \begin{equation*}
        \widetilde\omega(t_\varepsilon)\ge M_\varepsilon-\delta_\varepsilon>x+\delta_\varepsilon>\widetilde x,
    \end{equation*}
    which yields $\tau(\widetilde x,\widetilde\omega)\le t_\varepsilon<\tau+\varepsilon$. Combining these two inequalities yields $|\tau-\tau(\widetilde x,\widetilde\omega)|<\varepsilon$, and thus the continuity of $\tau$ at $(x,\omega)$.
\end{proof}

\begin{proposition}
    Let $X^1,\dots,X^N$ be independent symmetric simple random walks in $\ZZ$. Then 
    \begin{equation*}
        n^{-1/2}\R_n^k\distribution\lambda\left(\left\{x\in\RR, \tau_x^k\le 1\wedge\min_{j\ne k}\tau_x^j\right\}\right),
    \end{equation*}
    where $\tau_x^j\eqdef\inf\{t>0, W_t^j=x\}$ and $W^1,\dots,W^N$ are independent standard Brownian Motions in $\RR$. 

    Furthermore, joint convergence in $k\in[\![N]\!]$ also holds.
\end{proposition}

\begin{proof}
    For $K>0$, we introduce
    \begin{equation*}
        \R_n^{k,K}\eqdef\sum_{|x|\le K\sqrt n}\1\left(T_x^k\le n, T_x^k<\min_{\ell\ne k}T_x^\ell\right).
    \end{equation*}
    As we saw before, it is sufficient to show that 
    \begin{equation*}
        n^{-1/2}\R_n^{k,K}\distribution\lambda\left(\left\{|x|\le K, \tau_x^k\le 1\wedge\min_{j\ne k}\tau_x^j\right\}\right),
    \end{equation*}
    as the cutoff can be removed by letting $K\rightarrow\infty$. To do this, we recall 
    \begin{equation*}
        X_t^{j,n}\eqdef n^{-1/2}X_{\lfloor nt\rfloor}^j,\quad t\in[0,1],
    \end{equation*}
    and then by using the notation from Lemma \ref{lemma:joint continuity of Brownian hitting times} and introducing $x_n\eqdef n^{-1/2}\lfloor n^{1/2}x\rfloor$ for $x\in\RR$, we get
    \begin{align*}
        n^{-1/2}\R_n^{k,K}&=n^{-1/2}\sum_{|x|\le K\sqrt n}\1\left(\tau\left(n^{-1/2}x,X^{k,n}\right)\le 1\right)\\
        &\hspace{6em}\times\1\left(\tau\left(n^{-1/2}x,X^{k,n}\right)<\min_{\ell\ne k}\tau\left(n^{-1/2}x,X^{\ell,n}\right)\right)\\
        &=\int_{\mathcal B_K^{(n)}}\d x\ \1\left(\tau\left(x_n,X^{k,n}\right)\le 1,\tau\left(x_n,X^{k,n}\right)<\min_{\ell\ne k}\tau\left(x_n,X^{\ell,n}\right)\right)\\
        &=:\int_{\RR}\d x\ f_k\left(x_n,X^{1,n},\dots X^{N,n}\right)\1\left(|x_n|\le K\right).
    \end{align*}
    Let $\widetilde X^{j,n}$ be the linearly interpolated version of $X^{j,n}$. The fact that $X^j$ is a simple random walk implies that for any $y\in n^{-1/2}\ZZ$,
    \begin{equation*}
        \tau\left(y,X^{j,n}\right)=\tau\left(y,\widetilde X^{j,n}\right).
    \end{equation*}
    In particular, for any $x\in\RR$,
    \begin{equation*}
        f_k\left(x_n,X^{1,n},\dots X^{N,n}\right)=f_k\left(x_n,\widetilde X^{1,n},\dots,\widetilde X^{N,n}\right).
    \end{equation*}
    Donsker's Theorem yields
    \begin{equation*}
        \left(X^{1,n},\dots,X^{N,n}\right)\distribution\left(W^1,\dots,W^N\right)
    \end{equation*}
    in the product uniform topology on $\mathcal C^N$. By Skorokhod's representation Theorem, we may suppose without loss of generality that the previous convergence holds $a.s.$ instead of in distribution. By Lemma \ref{lemma:joint continuity of Brownian hitting times}, $f$ is continuous $\lambda\otimes\mathbb W^{\otimes N}$-a.e. on $\RR\times\mathcal C^N$, and so for $\lambda$-a.a. $x\in\RR$,
    \begin{equation*}
        f_k\left(x_n,\widetilde X^{1,n},\dots,\widetilde X^{N,n}\right)\underset{n\rightarrow\infty}{\overset{a.s.}\longrightarrow}f_k\left(x,W^1,\dots,W^N\right).
    \end{equation*}
    In particular, the dominated convergence theorem yields 
    \begin{align*}
        n^{-1/2}\R_n^{k,K}&\underset{n\rightarrow\infty}{\overset{a.s.}\longrightarrow}\int_{\RR}\d x\ f_k\left(x,W^1,\dots,W^N\right)\1\left(|x|\le K\right)\\
        &=\lambda\left(\left\{|x|\le K, \tau\left(x, W^j\right)\le 1, \tau\left(x,W^j\right)<\min_{\ell\ne k}\tau\left(x,W^\ell\right)\right\}\right)\\
        &\overset{a.s.}=\lambda\left(\left\{|x|\le K, \tau_x^j\le 1\wedge\min_{\ell\ne k}\tau_x^\ell\right\}\right),
    \end{align*}
    since for every $x\in\RR$, $\tau_x^j$ has a density. To see that the joint convergence in $k$ holds, one can simply use Skorokhod's representation and note that 
    \begin{equation*}
        (f_1,\dots,f_N)\left(x_n,\widetilde X^{1,n},\dots,\widetilde X^{N,n}\right)\underset{n\rightarrow\infty}{\overset{a.s.}\longrightarrow}(f_1,\dots,f_N)\left(x,W^1,\dots,W^N\right)
    \end{equation*}
    for $\lambda$-a.a. $x\in\RR$ using the same argument as before.
\end{proof}

In particular, we notice that the effect of the competition in the case of symmetric simple random walks in $\ZZ$ appears as early as the first order in the asymptotics of $\R_n^k$. We expect this to also be the case for any $\beta$-stable walk with $\beta>1$.

Finally, note that the previous results can all be extended to the case where the $X^k$ are still $\beta$-stable walks with the same scale function $b$ but potentially not identically distributed. In this case, we have $b(n)^{-1}X_n^k\distribution U_1^k$ for all $k\in[\![N]\!]$, where $U^1,\dots,U^N$ are independent $\beta$-stable processes in $\RR^d$. By letting $h_0(n)\eqdef\sum_{j=1}^n b(j)^{-d}$, then by (\ref{limit:trunctaed green's function asymptotics}) we have 
\begin{equation*}
    h_k(n)\underset{n\rightarrow\infty}\sim p_1^k(0)h_0(n),
\end{equation*}
where $h_k$ is the truncated Green's function associated to $X^k$ and where $p_t^k$ is the density of $U_t^k$. By writing $p_k\eqdef p_1^k(0)$, a direct adaptation of the results in Section \ref{section:Limit Theorems for R^(A->B) and I^(A->B)} shows that our LLN becomes 
\begin{equation*}
    \frac{h_0(n)}{n}\R_n^k\underset{n\rightarrow\infty}{\overset{L^2}\longrightarrow}(p_k)^{-1},
\end{equation*}
with $a.s.$ convergence if $s(n)\ge1$ for any $n\ge1$, and similarly our CLT becomes 
\begin{equation*}
    \label{convergence:CLT without iid}
    \frac{h_0(n)^2b(n)^d}{n^2}\left(\R_n^k-\EE\left[\R_n^k\right]\right)\distribution-\left((p_k)^{-2}\gamma^k+\sum_{j\ne k}(p_jp_k)^{-1}\overline\alpha^{j,k}\right)\left([0,1]_\le^2\right).
\end{equation*}
Note also that like before, the joint convergence in $k$ also holds.

\paragraph{Acknowledgements.} We thank Arthur Plaud and Olivier Bénichou for discussions. In particular, Arthur Plaud’s presentation at the IHP meeting on April 28, 2025 of their ongoing work on the competitive-range problem in the context of competitive foraging preceded our work on this topic and was both stimulating and instructive. This work was funded by the CNRS project MITI. The author wishes to warmly thank his advisors Vincent Bansaye, Jean-René Chazottes and Sylvain Billiard for the time spent discussing and proofreading this work. This work  was partially funded by the Chair “Modélisation Mathématique et Biodiversité of VEOLIA-Ecole Polytechnique-MNHN-F.X., by the European Union (ERC, SINGER, 101054787)

\printbibliography[title={Bibliography}]

\end{document}